\documentclass{article}
\usepackage[utf8]{inputenc}
\usepackage{authblk}

\usepackage{amsthm}
\usepackage{xcolor}
\usepackage{ulem}

\usepackage{latexsym,amsthm,amscd}
\usepackage{amsmath}
\usepackage{amssymb}
\usepackage{amsfonts}
\usepackage{tikz}
\usepackage{stmaryrd}
\usepackage{hyperref}
\usepackage[most]{tcolorbox}
\usetikzlibrary{patterns}
\usepackage{enumerate}
\usepackage{units}
\usepackage{CJKutf8}
\usepackage{ulem}
\usepackage{geometry}
\usepackage[capitalise]{cleveref}

\newcommand{\blue}{\color{blue}}

\newcommand\rst{\bgroup\markoverwith{\textcolor{red}{\rule[0.5ex]{2pt}{0.4pt}}}\ULon}

\title{Residual stratification and the Cantor-Bendixson structures of dual algebraic coframes}

\author[1]{Silvère Gangloff\footnote{Corresponding author}}
\affil[1]{
	University of Ostrava
    
    \url{silvere.gangloff@osu.cz}
	}

\author[2]{Alonso N\'u\~nez\footnote{Alonso N\'u\~nez is funded by the ANR Project \textit{``Ordinal Time Computations''} (ANR-24-CE48-0335).}}
\affil[2]{
\'Ecole Polytechnique

\url{anunez@lix.polytechnique.fr}
}

\date{\today}

\newtheorem{theorem}{Theorem}[section]
\newtheorem*{theorem*}{Theorem}
\newtheorem{lemma}[theorem]{Lemma}
\newtheorem{corollary}[theorem]{Corollary}
\newtheorem{question}[theorem]{Question}
\newtheorem{definition}[theorem]{Definition}
\newtheorem{example}[theorem]{Example}

\newtheorem*{problem*}{Problem}
\newtheorem{notation}[theorem]{Notation}
\newtheorem{proposition}[theorem]{Proposition}
\newtheorem{remark}[theorem]{Remark}

\begin{document}

\maketitle

\begin{abstract}
    We introduce the notion of \textit{residual derivative} for elements of a preordered set, a construction that generalizes both the Frattini subgroup in algebra and the Cantor-Bendixson derivative in T1 topological spaces. For dual algebraic coframes with topologies compatible with order, we establish a partial correspondence between 
    the Cantor-Bendixson structure of the lattice
    and the residual derivatives of its elements. 
    Within this framework, we provide a complete characterization of the first two Cantor-Bendixson levels in terms of the lattice’s residual structure. This provides a unified lens
    to study the Cantor-Bendixson structures of topological spaces across domains ranging from algebra to functional analysis and dynamics, facilitating the transfer of analytic techniques between them.
\end{abstract}

\tableofcontents

\section{Introduction}

\paragraph{Cantor-Bendixson structure}

The \textit{Cantor-Bendixson derivative} of a set $X$ relative to a topology $\tau$ on $X$ is the set $X'$ of elements of $X$ that are non-isolated for $\tau$. For every ordinal $\alpha$, we denote by 
$\mathcal{S}_{\alpha}(X,\tau)$ the $\alpha$th derivative of $X$, defined by transfinite induction as follows: $\mathcal{S}_{0}(X,\tau) := X$, for any non-limit ordinal $\alpha$, $\mathcal{S}_{\alpha+1}(X,\tau) := (\mathcal{S}_{\alpha}(X,\tau))'$, and for every limit ordinal $\lambda$, $\mathcal{S}_{\lambda}(X,\tau) := \bigcap_{\alpha < \lambda}\mathcal{S}_{\alpha}(X,\tau)$. The \textit{Cantor-Bendixson rank} of $X$, when it exists, is the smallest ordinal $\alpha$ such that $\mathcal{S}_{\alpha+1}(X,\tau) = \mathcal{S}_{\alpha}(X,\tau)$. We denote it by $\texttt{r}(X,\tau)$. Historically, this structure emerged from the work of Georg Cantor on a question posed by Riemann: whether a trigonometric series that converges to zero everywhere must have all its coefficients equal to zero. Cantor proved that if a series converges to zero outside a {``thin''} set 
{, that is, a set whose Cantor-Bendixson derivative vanishes after finitely many iterations}, then the coefficients of the series are zero. The idea of classifying such thin sets 
{led to the introduction} of transfinite numbers, 
{which quantify how many derivation steps are required before the process stabilises}. In this way, the Cantor-Bendixson process appears as one of the earliest instances of a transfinite construction in analysis. This line of work was further developed by Bendixson, who proved that every closed subset of $\mathbb{R}$ 
{admits a unique decomposition} into a perfect set and a countable set.
{The latter consists precisely of the points removed during the transfinite derivation} process. This result, now known as the Cantor-Bendixson theorem, was later extended to more general topological spaces such as Polish spaces, and became a foundational tool in descriptive set theory.

\paragraph{Applications of the Cantor-Bendixson process} More recently, the Cantor-Bendixson process has found applications across several areas of mathematics, including algebra, functional analysis, and dynamical systems. In algebra, it appears naturally in the study of spaces of subgroups endowed with the Chabauty topology. In this setting, it was used to 
classify the topological spaces that arise as the space of subgroups of a countable abelian group \cite{CornulierGuyotPitsch_SubgroupsAbelian}, and 
{studied as an object of independent interest}. The Cantor-Bendixson rank was computed for the space of sugroups of Grigorchuk group and of the Gupta-Sidki 3-group \cite{skipper2020cantorbendixsonrankgrigorchukgroup} and virtually metabelian groups with high Cantor-Bendixson ranks were constructed \cite{Cornulier_CBRank_Metabelian}. 
In functional analysis, the Cantor-Bendixson structure of countable compact Hausdorff spaces governs the structure of the associated function spaces. The Banach space $\mathcal{C}(K,\mathbb R)$, the set of continuous functions from $K$ to $\mathbb R$, is determined, up to isomorphism, by the Cantor-Bendixson rank of $K$ together with the cardinalities of its successive derived layers, leading to classification results for spaces of continuous functions \cite{BessagaPelczynski1960}. 
In model theory, it is used to prove, for instance, that certain models are almost prime \cite{BazhenovMarchuk2023}.




\paragraph{Multidimensional symbolic dynamics and the hyperspace of shifts} In symbolic dynamics, the Cantor-Bendixson process has been used to analyze the structure of multidimensional shifts \cite{ballier2013structuringmultidimensionalsubshifts}, connecting it to a pre-order on configurations that compares their finite parts and to the notion of complexity. It has also been studied in the context of countable shifts, through the realization of high Cantor-Bendixson ranks \cite{salo2013constructionscountablesubshiftsfinite}. More recently, we have proved that 
{in sharp contrast to the case} $d=1$ \cite{pavlov2022structuregenericsubshifts} , the hyperspace of $d$-dimensional shifts with the Hausdorff topology has infinite Cantor-Bendixson rank when $d \ge 2$ \cite{GN25}.
In the same work, we also characterized the isolated points in terms of maximal subsystems---a purely order-theoretic notion.

\paragraph{The residual derivative}
In the present document, we generalize this perspective to dual algebraic coframes equipped with topologies compatible with the order. This includes the Lawson topology and the Hausdorff topology on the hyperspace of shifts. {To capture the phenomenon that Cantor-Bendixson layers can be expressed in purely order-theoretic notions at an abstract level, we introduce a new order-theoretic operation: the \textit{residual derivative}.}
While it can be defined on any complete lower semilattice, 
{the residual derivative finds a natural setting in dual algebraic coframes. In this context, it provides a powerful}
tool to investigate the Cantor-Bendixson process. In particular, it allows us to characterize the second layer of the Cantor-Bendixson process, that is, the set of elements of the poset that become isolated once isolated elements have been removed. 
Consider a complete lower semilattice $(L,\le)$. 
Let $x$ be an element of $L$. We denote by $\mathcal{M}(x)$ the set of \textit{maximal subelements} of $x$, that is, the maximal elements 
{of the poset \(\{z\in L\mid z<x\}\)}. The residual derivative of $x$ is 
{defined by:}
\[\mu(x) = \bigwedge_{z \in \mathcal{M}(x)} z,\]
{whenever \(\mathcal{M}(x)=\emptyset\), and by \(\mu(x)=x\) otherwise.}

When $L$ is the set of closed subsets of a T1 space, the residual derivative coincides with the Cantor-Bendixson derivative. When $L$ is the set of subgroups of a group $G$, the derivative coincides with the Frattini sugroup 
$\Phi(G)$ of $G$ (which is originally defined as the intersection of all maximal subgroups of $G$). By analogy with the Cantor-Bendixson structure, we also define the residual rank and the residual kernel, leading to a whole conceptual structure and a battery of natural results presented in Section \ref{section.decomposition}. In particular, we prove that every element can be decomposed into its residual core and the co-Heyting subtractions of its maximal elements, which we call \textit{residues}. Every element that is smaller than $x$ and larger than its core can be decomposed into the core of $x$ and 
the elements of its \textit{residual poset}, which are the completely co-irreducible subelements of $x$ that are not smaller than its residual core. We show that these elements are exactly the ones obtained by taking residues iteratively. We also show that the residual derivative is a $\vee$-homomorphism.
Our main results exhibit a deep relation between the Cantor-Bendixson structure of any dual algebraic coframe and the residual structure of its elements by characterizing exactly the first two Cantor-Bendixson layers in terms of the residual structure. 
The characterizations of the first and second layers can be found in Section \ref{section.lattice.proofs.1} and in Section \ref{section.lattice.proofs.2}, respectively.

Our results are reminiscent of the work of Simmons \cite{Simmons_NearDiscreteness}, who proved that on the family of near-discrete lattices, Gabriel derivative and Cantor-Bendixson derivative coincide. However, our results differs from Simmons' in the sense that we do not prove an identity but a structured relation between two different derivatives. 





\paragraph{Significance} During the last decade, following the work of Mike Hochman \cite{HochmanMeyerovitch2010} \cite{Hochman2009}, establishing a strong relation between multidimensional symbolic dynamics and computability theory, most of the research efforts in the field has been devoted to deepen this relation. However multidimensional symbolic systems are still not well understood, and we believe that more conceptual structure is required for this understanding. We hope that this research direction will contribute to this. 
Recently, Van Cyr, Bryna Kra and Scott Schmieding \cite{CyrKraSchmieding_ChaoticAlmostMinimal} introduced the notion of chaotic almost minimal shifts, a class of shifts that includes Rudolph's discrete analogue of the $\times 2 \times 3$ system on the circle \cite{Rudolph1990}. They showed that Furstenberg's $\times 2 \times 3$ conjecture fails within this class.

Our framework provides a different perspective on systems related to Rudolph's shift, since Rudolph's example belongs to the second layer of the Cantor--Bendixson structure of the hyperspace of two-dimensional shifts. Studying its position within this broader topological framework, and comparing it with other systems that occupy the same layer, may help reveal some of the structural features underlying Rudolph's shift.


More broadly, although many natural topologies on ordered structures are closely tied to the underlying order and can be analysed through structural correspondences, the Cantor--Bendixson process is defined independently of any order-theoretic considerations. One of the main contributions of this work is to show that, in a broad class of ordered topological structures, the Cantor--Bendixson hierarchy can nevertheless be recovered from a canonical order-theoretic operation, namely the residual derivative. This reveals a previously unnoticed connection between topological and order-theoretic stratifications and provides a common framework for studying Cantor--Bendixson structures in diverse mathematical contexts.

\bigskip

The remainder of this text is structured as follows. We provide some background on {dual algebraic coframes and order-compatible topologies on posets} 
in Section \ref{section.setup} and detail some examples. Section \ref{section.cb.defs} defines the residual derivative and Section \ref{section.decomposition} presents the residual structure that follows from this definition. Section \ref{section.lattice.proofs.1} and Section \ref{section.lattice.proofs.2} contain the characterizations of the first two layers of the Cantor-Bendixson structures of dual algebraic coframes, 
and Section \ref{section.comments} concludes with open questions.

\section{\label{section.setup} Dual algebraic coframes and topology}

In this section, we recall definitions of order theory in Section \ref{section.background.lattices}, and of dual algebraic lattices and coframes in Section \ref{section.def.densely.dominated}. We also define a notion of order-compatible topology on a poset in Section \ref{section.def.order.compatible}, and detail examples of interest in Section \ref{section.ex.densely.dominated}.

\subsection{Background\label{section.background.lattices}}

A \textit{partially ordered set} (poset) is a pair $(L,\le)$ where $L$ is a set and $\le$ is a partial order. 
Given a poset $(L,\le)$, for all $x \in L$, we denote by $\mathop{\downarrow}x$ the set $\{z \in L : z \le x\}$ and by $\mathop{\uparrow}x$ the set $\{z \in L : x \le z\}$. 
We say that a poset $(L,\le)$ is a \textit{lower (resp. upper) semilattice} when for all $x,y \in L$, the set $\{z \in L : z \le x , z \le y\}$ (resp. $\{z \in L : x \le z , y \le z\}$) admits a maximum (resp. minimum) denoted by $x \wedge y$ (resp. $x \vee y$). We say that a lower (resp. upper) semilattice is \textit{complete} when for every subset $S$ of $L$, the set $\{x \in L : \forall s \in S, x \le s\}$ has a maximum (resp. the set $\{x \in L : \forall s \in S, s \le x\}$ has a minimum), denoted by $\bigwedge S$ or $\bigwedge_{s \in S} s$ (resp. $\bigvee S$ or $\bigvee_{s \in S}s$). A poset is called a \textit{lattice} when it is both a lower semilattice and an upper semilattice. A lattice is said to be lower (resp. upper) complete if it is complete as a lower (resp. upper) semilattice. A lattice $(L,\le)$ is said to be \textit{distributive} when for all $x,y,z \in L$, $x \wedge (y \vee z) = (x \wedge y) \vee (x \wedge z)$.

\begin{lemma}\label{lemma.upper.complete.ideals}
    In a 
    lower complete lattice $(L,\le)$, for every $x \in L$, $(\mathop{\downarrow}x,\le)$ is upper complete.
\end{lemma}

\begin{proof}
    For all $S \subset \mathop{\downarrow}x$, we prove that the set 
    $S^+ := \{x \in L : \forall s \in S, x \ge s\}$ has a minimum. Since the lattice $(L,\le)$ is lower complete, we know that $S^{++} := \{x \in L : \forall s \in S^+, x \le s\}$ has a maximum, denoted by $\bigwedge S^+$. Since by definition, $S \subset S^{++}$, we have $\bigwedge S^+ \in S^+$. By definition, $\bigwedge S^+$ is a minimum of $S^+$.
\end{proof}


A \textit{directed set} is a poset $(D,\le)$ in which any two elements admit a common upper bound; that is, given $d,d' \in D$, there exists $d'' \in D$ such that $d,d' \le d''$. Dually, a \textit{filtered set} is a poset $(F,\le)$ in which any two elements admit a common lower bound; that is, given $f,f' \in F$, there exists $f'' \in F$ such that $f'' \le f$ and $f'' \le f'$. A \textit{net} (resp. \textit{filter}) is a function whose domain is a directed (resp. filtered) set. Observe that the image of an increasing (resp. non-increasing) net is a directed (resp. filtered) set, whereas the image of an increasing (resp. non-increasing) filter is a filtered (resp. directed) set.

\subsection{Dual algebraic coframes\label{section.def.densely.dominated}}

In this section, we recall the definition of dual algebraic coframe and prove that the dually compact elements of a dual algebraic coframe form a lower semilattice.


\begin{definition}
    Consider a poset $(L,\le)$. An element $x$ of $L$ is said to be \textbf{dually compact} when for all filtered set $F$ in $L$ which has an infimum $\bigwedge F \le x$, then there exists $f \in F$ such that $f \le x$. We denote by $\mathcal{K}(L,\le)$ the set of dually compact elements of $(L,\le)$.
\end{definition}

\begin{definition}
    A poset $(L,\le)$ is called a \textbf{dual algebraic} lattice when it is a lower complete lattice and for every $x \in L$ there exists a filtered set $F_x$ in $\mathcal{K}(L,\le)$ such that $\bigwedge F_x = x$.
\end{definition}

\begin{definition}
    {A poset $(L,\le)$ is called a \textbf{coframe} when it is an upper lattice, a lower complete lattice, and satisfies the dual infinite distributivity: for all $x \in L$ and every subset $S$ of $L$, $x \vee  \bigwedge_{s \in S} s = \bigwedge_{s \in S} (x \vee s)$.}
\end{definition}

{Note that every coframe is a distributive lattice.}
{In the following, we will call \textbf{dual algebraic coframe} a poset that is both a dual algebraic lattice and a coframe.}

\begin{lemma}\label{lemma.loc.dominant.semilattice}
    {For every dual algebraic coframe $(L,\le)$,  the ordered set $(\mathcal{K}(L,\le),\le)$ is a lower semilattice.}
\end{lemma}

\begin{proof}
{Fix \(x,z\in\mathcal{K}(L,\le)\). We show that \(x\wedge z\in\mathcal{K}(L,\le)\). Let \(F\) be a filtered set with \(\bigwedge F \le x\wedge z\). Since \(x\wedge z\leq x\), we have that \[x=x\vee (x\wedge z)=x\vee\bigwedge F.\]
As $(L,\le)$ is a coframe, we have 
\[x = \bigwedge_{f\in F}(x\vee f).\]
The dual compactness of \(x\) then implies that \(x \in \{x\vee f\mid f\in F\}\). Hence, there exists \(f_x\in F\) such that \(x\vee f_x=x\), so \(f_x\leq x\). Symmetrically, there exists $f_z \in F$ such that $f_z \le z$. Since \(F\) is filtered, there exists \(f\in F\) with \(f\leq f_x\) and \(f\leq f_z\), so that \(f\leq x\wedge z \le \bigwedge F\), hence, \(x\wedge z=f\in F\).}
\end{proof}

For further reading on algebraic lattices and frames, one can consult \cite{Gierz2003}.

\subsection{\label{section.def.order.compatible}{Order-compatible topologies}}

Given a directed set $(D,\le)$ and a topological space $(X,\tau)$, we say that a net $\boldsymbol{x} : D \rightarrow X$ converges to a point $x \in X$ when for every $o \in \tau$ with $x \in o$, there exists $d \in D$ such that for all $d' \ge d$, $\boldsymbol{x}_{d'} \in o$.

\begin{definition}\label{def.comp.top}
Let $(L,\leq)$ be a poset. We say that a topology $\tau$ on $L$ is \textbf{order-compatible} when the following conditions hold:
\begin{itemize}
\item[(i)] every non-increasing (resp. non-decreasing) net $\boldsymbol{x}:D\to L$ for which $\bigwedge_{d\in D} \boldsymbol{x}_d$ (resp. $\bigvee_{d\in D} \boldsymbol{x}_d$) exists converges to it;
\item[(ii)] the map \(\vee : L \times L \to L\) is continuous with respect to $\tau$ and $\tau \times \tau$;
\item[(iii)] the order $\le$ is closed as a subset of $L \times L$.
\end{itemize}
\end{definition}

The notion of order-compatible topology is naturally related to multiple classical ideas at the interface of order and topology. Condition \((ii)\) sets this structure within the framework of topological semilattices \cite{Gierz2003}. The closedness of the order in Condition \((iii)\) is a standard in ordered topological spaces, in particular, in the sense of Nachbin \cite{NachbinLeopoldo1965Tao}. We notice that, under the additional assumption that the space is Hausdorff, Condition \((iii)\) is implied by Condition \((ii)\): since \(x\leq y\) if and only if \(x\vee y=y\), the order is the equalizer of the continuous maps \((x,y)\mapsto x\vee y\) and \((x,y)\mapsto y\) and is therefore closed whenever the topology is Hausdorff. Thus, Condition \((iii)\) is relevant in non-Hausdorff spaces. Finally, Condition \((i)\) is a compatibility between topological and order-theoretic convergence which is involved in the definition of  Lawson topology \cite{Gierz2003}.
To the best of our knowledge, these three conditions have not been combined into a single definition in the existing literature. A canonical source of examples is the class of compact Hausdorff topological semilattices. \cref{def.comp.top} then provides a more general framework than the aforementioned class, as it has been designed to handle non-Hausdorff or non-compact situations. 

\begin{lemma}
Let \((L,\leq)\) be a dual algebraic lattice and let \(\tau\) be an order-compatible topology on $L$. Then, the topological space $(L,\tau)$ is Hausdorff.
\end{lemma}

\begin{proof} 
    Consider two distinct elements $z,z' \in L$. Assume that $z \not \le z'$ and $z' \not \le z$ (the cases $z < z'$ and $z' < z$ are processed similarly). 
    {Since $(L,\le)$ is dual algebraic}, there exist 
    two non-increasing nets 
    {$\boldsymbol{z} : D' \rightarrow \mathcal{K}(L,\le)$} and 
    {$\boldsymbol{z}' : D \rightarrow \mathcal{K}(L,\le)$} such that 
    {$\bigwedge_d \boldsymbol{z}_d = z$} and 
    {$\bigwedge_{d'} \boldsymbol{z}'_d = z'$}. As a consequence of the assumption, there exist $d,d'$ such that $z' \notin \mathop{\downarrow}\boldsymbol{z}_d$ and $z \notin \mathop{\downarrow}\boldsymbol{z}'_d$. Since $\le$ is closed, for all $x$, $\mathop{\downarrow} x$ is closed. Since $\boldsymbol{z}_d$ and $\boldsymbol{z}'_{d'}$  are dually compact, 
    $\mathop{\downarrow}\boldsymbol{z}_d$ and $\mathop{\downarrow}\boldsymbol{z}'_{d'}$ are open, 
    and thus the sets
    $\mathop{\downarrow}\boldsymbol{z}_d \setminus \mathop{\downarrow}\boldsymbol{z}'_{d'}$ and $\mathop{\downarrow}\boldsymbol{z}'_{d'} \setminus \mathop{\downarrow}\boldsymbol{z}_d$ are both open and separate the two elements $z$ and $z'$.
\end{proof}

Given a set $S$, a choice function of $S$ is a function $\theta: \mathcal{P}(S) \setminus \{\emptyset\} \rightarrow S$ such that for every $P \in \mathcal{P}(S)$, it holds that $\theta(P) \in P$. 

\begin{lemma}\label{lemma.finite.type.included}
    {For all dual algebraic coframe $(L,\le)$ such that $\mathcal{K}(L,\le)$ has a choice function and all order-compatible topology $\tau$ on $L$, an element $x \in L$ is dually compact if and only if $\mathop{\downarrow}x$ is open.}
\end{lemma}

\begin{proof}
   $(\Leftarrow)$ 
   {For any filtered set $F$ with $\bigwedge F \le x$, since $\mathop{\downarrow}x$ is open and $F$ converges to $\bigwedge F$, there is some $f \in F$ such that $f \le x$. We prove that $x$ is dually compact.} $(\Rightarrow)$ {Assume that $\mathop{\downarrow}x$ is not open, that is, there exists $z \in \mathop{\downarrow}x$ such that for every $o \in \tau$ with $z \in o$, the set $o \setminus (\mathop{\downarrow}x)$ is not empty. Since $(L,\le)$ is a dual algebraic lattice, there is some filtered set $F_z$ in $\mathcal{K}(L,\le)$ with $\bigwedge F_z = z$. 
Using that  $\mathop{\downarrow}f$ is open for all \(f\in F\), we have that none of the elements of $F$ is in $\mathop{\downarrow}x$. From the topology being order-compatible and the convergence of $F_z$ to $z$ (that is, $F_z \cap o \neq \emptyset$), it follows that $o \setminus (\mathop{\downarrow}x)$ has an element in $\mathcal{K}(L,\le)$. The choice function \(\theta\) on $\mathcal{K}(L,\le)$ allows us to set $z_o := \theta(\mathcal{K}(L,\le) \cap(o\setminus \mathop{\downarrow}x ))$. Thenm the set $\tau_x := \{o \in \tau: x \in o\}$ is a filtered set for the order $\subset$. We thus have a filter $\boldsymbol{z} : \tau_x \rightarrow L$ such that for any $o \in \tau_x$, $\boldsymbol{z}_o \in o \setminus (\mathop{\downarrow}x)$. 
    By definition, this filter converges to $z$. By continuity of $\vee$, we also have $x \vee \boldsymbol{z}_o \rightarrow x \vee z = x$. Since $\boldsymbol{z}_o \not \le x$, for all $o$, we have $x \vee \boldsymbol{z}_o > x$. This means that $\bigwedge_{o \in \tau_x} (x \vee \boldsymbol{z}_o) = x$, hence $x$ is not dually compact.}
\end{proof}

\begin{remark}
The equivalence from Lemma~\ref{lemma.finite.type.included} provides a direct translation between order-theoretic compactness and a topological property of principal ideals. In particular, the family of dually compact elements can be recovered purely from the topology as those elements whose lower sets are open. On the other hand, observe that since $\le$ is closed, $\mathop{\downarrow}x$ is always a closed set.
\end{remark}



\subsection{Examples\label{section.ex.densely.dominated}}

In this section, we present some examples of 
{order-compatible topologies on dual algebraic coframes}. 

\paragraph{Lawson topology}

\begin{definition}
    {Consider a poset $(L,\le)$.} We call \textbf{dual Lawson topology} on $L$ the topology 
    $\tau^*(L,\le)$ generated by the subbase 
    \[\left\{\mathop{\downarrow} x : x \in \mathcal{K}(L,\le)\right\} \cup \left\{L \setminus (\mathop{\downarrow} x) : x \in \mathcal{K}(L,\le)\right\}.\]
\end{definition}

We prove here that {when $(L,\le)$ is a dual algebraic lattice,} 
{$\tau^*(L,\le)$ is order-compatible} (stated as Proposition \ref{proposition.lawson.dominated}).

\begin{lemma}\label{lemma.conv.filtered.set}
    {Consider a poset $(L,\le)$.} For all filtered set $F$ in $L$ such that $\bigwedge F$ exists, $F$ converges to $\bigwedge F$ for the dual Lawson topology. For all directed set $D$ in $L$ such that $\bigvee D$ exists, $D$ converges to $\bigvee D$ for the dual Lawson topology.
\end{lemma}

\begin{proof}
    Fix some filtered set $F$. For all $x \in \mathcal{K}(L,\le)$ such that $\bigwedge F \le x$, there exists $f_x \in F$ such that $f_x \le x$.
    As a consequence, for all $f \in F$ such that $f \le f_x$, by transitivity, we have $f \le x$. Additionally, for any $x \in L$, if $\bigwedge F \not \ge x$ then there exists $f'_x \in F$ such that $f'_x \not \ge x$. Therefore, for all $f \in F$ such that $f \le f'_x$, we have $f \not \ge x$, otherwise we would have $f'_x \ge x$ by transitivity.     
    These two facts imply that $F$ converges to $\bigwedge F$ for the dual Lawson topology. The proof is analogous for the case of directed sets and the convergence of their join.
\end{proof}

\begin{lemma}\label{lemma.vee.continuous}
    {Consider a poset $(L,\le)$. The operator $$\vee : (L \times L, \tau^*(L,\le) \times \tau^*(L,\le)) \rightarrow (L,\tau^*(L,\le)),$$ is continuous.}
\end{lemma}

\begin{proof}
    {For $z \in L$, 
    $\vee^{-1}(\mathop{\downarrow}z) = \{(x,y) \in L \times L : x \vee y \le z \}$. 
    We thus have $\vee^{-1}(\mathop{\downarrow}z) = (\mathop{\downarrow}z)^2$, which is open for the topology $\tau^*(L,\le) \times \tau^*(L,\le)$ when $z$ is dually compact. We also have $\vee^{-1}(L \setminus \mathop{\downarrow}z) = ((L \setminus \mathop{\downarrow}z) \times L ) \cup (L \times (L \setminus \mathop{\downarrow}z))$, and when $z$ is dually compact,  since $L \setminus \mathop{\downarrow}z$ is open for the dual Lawson topology, $\vee^{-1}(L \setminus \mathop{\downarrow}z)$ is open. We just proved that $\vee$ is continuous.}
\end{proof}
\begin{lemma}\label{lemma.order.closed}
    {Let $(L,\le)$ be a dual algebraic lattice. The order $\le$ is closed for the dual Lawson topology.}
\end{lemma}

\begin{proof}
{Set $R:=\{(x,z) \in L \times L : x \not \le z \}$. It is sufficient to see that $R$ is open. Observe that for all $(x,z) \in R$, there exists $f \in \mathcal{K}(L,\le)$ such that $x \le f$ and $f \not \le z$, so that $x \in \mathop{\downarrow}f$ and $z \in L \setminus (\mathop{\downarrow}f)$. The set $(\mathop{\downarrow}f) \times (L \setminus (\mathop{\downarrow}f))$ is open, contains $(x,z)$ and is included in $R$.}
\end{proof}

As a consequence of the three preceding lemmas, we immediately obtain the following.


\begin{proposition}\label{proposition.lawson.dominated}
    {For all dual algebraic lattice $(L,\le)$, $\tau^*(L,\le)$ is order-compatible.}
\end{proposition}

\begin{remark}
    An example of dual algebraic coframe $(L,\le)$ such that $\mathcal{K}(L,\le)$ has a choice function is the lattice of subgroups of the additive group $(\mathbb Q, +)$ with order $\supset$, for instance, as its dually compact elements are the finitely generated subgroups, forming a countable set. 
\end{remark}

\paragraph{Hausdorff topology on the hyperspace of shifts}

For $d \ge 1$, a shift of dimension $d$ is a subset of the space $\mathcal{A}^{\mathbb Z^d}$, for some finite set $\mathcal{A}$. This space is closed for the infinite product of the discrete topology, and invariant by the shift action. A pattern on alphabet $\mathcal{A}$ is any element of some $\mathcal{A}^{\mathbb U}$, where $\mathbb U$ is a finite subset of $\mathbb Z^d$. A pattern $p \in \mathcal{A}^{\mathbb U}$ is said to appear in a configuration $x \in \mathcal{A}^{\mathbb Z^d}$ when there exists $\textbf{u} \in \mathbb Z^d$ such that $x_{\textbf{u}+\mathbb U} = p$.
For a shift $X$, the language $\mathcal{L}(X)$ of $X$ is the set of patterns that appear in at least one configuration of $X$. 
For a finite set of patterns $\mathcal{F}$, it is typically denoted by $X_{\mathcal{F}}$ the set of configurations in which no element of $\mathcal{F}$ appear. A shift $X$ is said to be of finite type when there exists a finite set of patterns $\mathcal{F}$ such that $X = X_{\mathcal{F}}$.

We denote by $\mathcal{H}^d$ the set of $d$-dimensional shifts, and by $\tau^d_{\mathcal{H}}$ the Hausdorff metrics topology on $\mathcal{H}^d$. We call $(\mathcal{H}^d,\subset,\tau_{\mathcal{H}}^d)$ the \textit{hyperspace} of $d$-dimensional shifts. Its 
{dually compact} elements are the shifts of finite type, and for any non-finite type shift $X$ there exists a decreasing sequence of shifts of finite type $(X_n)_n$ such that $X = \bigcap_n X_n$. {Thus, \((\mathcal{H}^d,\subset)\) is a dual algebraic lattice}. The set $\mathcal{K}(\mathcal{H}^d,\subset)$ is countable (as the set of finite sets of patterns is countable), so it has a choice function. The lattice $(\mathcal{H}^d,\subset)$ is clearly 
lower complete {and satisfies the dual infinite distributivity}. It is thus a coframe. It is also clear that 
{the Hausdorff topology is order-compatible}. 
    We conjecture that this example is a particular case 
    of the dual Lawson topology, but we do not prove this here.

\section{\label{section.cb.defs}Residual derivative}

In this section, we introduce the \textit{residual derivative}, an operator that associates to each element the meet of its maximal proper subelements. 
Iterating this operator yields a transfinite decreasing sequence, whose stabilization defines a rank and a canonical core. This provides an intrinsic stratification of the lattice, independent of any external topological construction. The derivative and the resulting stratification, introduced in Section~\ref{section.decomposition} , 
play a central role in our description of the Cantor-Bendixson structure of dual algebraic coframes. 


\subsection{Maximal subelements\label{section.def.max.subelements}}
{The residual derivative defined in the next subsection relies on the notion of maximal subelements. We define these objects and prove some of their basic properties in this subsection.}
Throughout this subsection, let $(L,\le)$ be a partially ordered set and let $H$ be a subset of $L$. Recall that an element $z$ of a subset $S\subseteq L$ is said to be \textbf{maximal} if, for every $z' \in S$, the relation $z' \ge z$ implies $z' = z$.

{
\begin{definition}
Fix some \(x\in L\). A \textbf{maximal \(H\)-subelement} of \(x\) is a maximal element of \(H\cap(\mathop{\downarrow}x\setminus\{x\})\). We denote by \(\mathcal{M}_H(x)\) the set of maximal \(H\)-subelements of \(x\). An \textbf{\(H\)-outcast} of \(x\) is an element of \( H\cap(\mathop{\downarrow}x\setminus\{x\})\) which is not below any maximal \(H\)-subelement of \(x\).
\end{definition}}

\begin{notation}
For all $n \in \mathbb N \cup \{\infty\}$, we denote by $\mathcal{T}_n(L,\le)$ the set of elements $x$ of $L$ such that $\mathcal{M}(x)$ has cardinality $n$. 
\end{notation}

\begin{remark}In most applications, \(H\) will be either \(L\) or \(\mathcal{T}_0(L,\leq)\). \end{remark}

The following elementary lemmas are used repeatedly throughout this text.
They state that maximal subelements interact naturally with the lattice operations and preserve dual compactness.





\begin{lemma}\label{lemma.cover.max}
    When $(H,\le)$ is an upper semilattice, for all $y, z \in \mathcal{M}_H(x)$ such that $y \neq z$, we have $y \vee z = x$. 
\end{lemma}

\begin{proof}
    Since $y \neq z$, we have $y \vee z > z$ or $y \vee z > y$. Because $y$ and $z$ are in $\mathcal{M}_H(x)$, we must have $y \vee z = x$.
\end{proof}

\begin{lemma}\label{lemma.max.of.max}
    Suppose that $L$ is a distributive lattice and that $(H,\le)$ is a lattice. For $y, z \in \mathcal{M}_H(x)$, $y \neq z$, the element $y \wedge z$ is a maximal $H$-subelement of $z$ and $y$.
\end{lemma}

\begin{proof}
    Consider some $w \in H$ such that $y \wedge z \le w \le z$. 
We have $y \le w \vee y \le x$. Since $H$ is an upper semilattice, 
$w \vee y \in F$. As $y$ is maximal, we have either $w \vee y = y$ or $w \vee y = x$. 
In the first case, $w \le y$, which implies $w \le z \wedge y$, and thus $w = y \wedge z$. 
In the second case, by distributivity, we have $(w \wedge z) \vee (y \wedge z) = z$. It follows then from $w \ge y \wedge z$ that $z = w \vee (w \wedge z) = w$. Since $H$ is a lower semilattice, $y \wedge z \in H$. Thus it is a maximal $H$-subelement of $z$. We then apply this to $y$ as well.
\end{proof}



\begin{lemma}
    {Given a dual algebraic coframe $(L,\le)$, for every $x \in \mathcal{K}(L,\le)$ and every $m \in \mathcal{M}(x)$, $m \in \mathcal{K}(L,\le)$.}
\end{lemma}

\begin{proof}
{Let \(m\in\mathcal{M}(x)\), and let \(F\) be a filtered set with \(\bigwedge F=m\). We show that \(m\in F\). Suppose, for the sake of contradiction, that \(m\notin F\), so that \(m < f\) for all \(f\in F\). Set \(F':=\{m\vee f\mid f\in F\}\). Since \(F\) is filtered, \(F'\) is filtered. For each \(f\in F\), since \(m < f\) we have \(m<m\vee f\). Since $(L,\le)$ is a coframe, 
\[
\bigwedge F'=\bigwedge_{f\in F}(m\vee f)=m\vee\bigwedge F=m\vee m=m.
\]
In particular, since $x$ is dually compact and $m = \bigwedge F' \le x$, by definition of dually compactness, there is some $f \in F$ such that $m \vee f \le x$.
There must be such an $f$ that also satisfies $m \vee f < x$. Otherwise we would have $\bigwedge F' = x = m$, which is not possible. This contradicts the hypothesis $m \in \mathcal{M}(x)$.}
\end{proof}

\subsection{Definition of the residual derivative\label{section.fratt.derivative}}

{The concept of this article is the residual derivative. We begin our discussion by formally defining it and some natural structures that emerge from it.}



\begin{definition}
    Let $(L,\le)$ be a complete lower semilattice, and $H$ a family of elements of $L$. For all $x$ in $L$, we call \textbf{residual $H$-derivative} of $x$ the element $\mu_H(x)$ defined as follows. When $\mathcal{M}_H(x) \neq \emptyset$: 
    \[\mu_H(x) := \underset{z \in \mathcal{M}_H(x)}{\bigwedge} z.\]
    Otherwise $\mu_H(x) := x$.
\end{definition}

Iterating the residual derivative yields the residual derivative sequence, defined formally as follows.


\begin{definition}
Let \((L,\leq)\) be a complete lower semilattice, let \(H\) be a family of elements of \(L\), and let \(x\in L\). The \textbf{residual \(H\)-{derivative sequence} 
of \(x\)} is the transfinite sequence \(\left(x_H^{(\alpha)}\right)_\alpha\)
defined recursively by
\[x_H^{(0)}:=x,\quad x_H^{(\alpha+1)}=\mu_H\!\left(x_H^{(\alpha)}\right),\]
for every ordinal \(\alpha\), and
\[x_H^{(\lambda)}=\bigwedge_{\alpha<\lambda}x_H^{(\alpha)},\]
for every limit ordinal \(\lambda\).
\end{definition}


This sequence is always non-increasing. Under mild assumptions, {it is ultimately constant}: 

\begin{theorem}\label{thm:existence.rank}
Let $(L,\le)$ be a dual algebraic lattice and suppose that $\mathcal{K}(L,\le)$ has a choice function. For every \(x\) in \(L\), there exists an ordinal $\alpha$ such that 
\[x_H^{(\alpha)} = x_H^{(\alpha+1)}.\] 
\end{theorem}

\begin{proof}
By definition, for all $x\in L$, \(\mathcal{M}_H(x)\) is contained in \(\mathop{\downarrow}x\), so \(\mu_H(x) \le x\). In particular, 
for all ordinal $\alpha$, \(x_H^{(\alpha+1)}\leq x_H^{(\alpha)}\), and for all limit ordinal \(\lambda\), by definition, we have \(x_H^{(\lambda)}=\bigwedge_{\beta<\lambda}x_H^{(\beta)}\leq x_H^{(\beta)}\) for all \(\beta<\lambda\). Hence, \(\alpha\mapsto x^{(\alpha)}\) is a non-increasing net. Suppose ad absurdum that for every ordinal \(\alpha\), we have \(x_H^{(\alpha+1)}<x_H^{(\alpha)}\). 
{Denote by $\theta$ a choice function 
of $\mathcal{K}(L,\le)$. For all ordinal $\kappa$, we thus have an injection from $\kappa$ to $\mathcal{K}(L,\le)\cap (\mathop{\downarrow} x)$ which to every ordinal $\alpha$ associates 
$\theta(\{f \in \mathcal{K}(L,\le) : f \ge x_H^{(\alpha+1)}, f \not \ge x_H^{(\alpha)}\})$.}
This straightforwardly contradicts Hartogs' lemma, which implies that there is an ordinal $\kappa$ such that there is no injection from $\kappa$ to $\mathcal{K}(L,\le)\cap (\mathop{\downarrow} x)$.
\end{proof}

{Under the conditions of Theorem \ref{thm:existence.rank}, we can define analogs of the Cantor-Bendixson rank and perfect kernel.}

\begin{definition}
Let $(L,\le)$ be a dual algebraic lattice and suppose that $\mathcal{K}(L,\le)$ has a choice function. For all \(x\in L\), the least ordinal \(\alpha\) such that \(x_H^{(\alpha)}=x_H^{(\alpha+1)}\) is called the \textbf{\(H\)-residual rank} of \(x\) and is denoted by \(\texttt{r}_H(x)\). The element \(\texttt{c}_H(x):=x_H^{(\texttt{r}_H(x))}\) is called the \textbf{\(H\)-core} of \(x\).
\end{definition}

\begin{remark}
The proof of Theorem \ref{thm:existence.rank} yields more than the existence of a stabilization ordinal. For every \(x\in L\), there is an injection from \(r_H(x)\) into \(\mathcal K(L,\le)\). Consequently, if \(\mathcal K(L,\le)\) is countable, then every element of \(L\) has countable \(H\)-residual rank.
\end{remark}

\subsection{Examples\label{section.examples.frattini}}

\begin{example}[Jacobson radical]
    Let $R$ be a ring. The \textit{Jacobson radical} of $R$ is the intersection of the maximal left ideals of $R$. The set of left ideals of $R$ form a lower semilattice for the inclusion preorder $\subset$, and for all ideals $I,J$, the meet of $I$ and $J$ is their intersection.
    The ring $R$ itself is a left ideal, and its residual derivative is exactly its Jacobson radical. See \cite{Lam1991} for more details. 
\end{example}

\begin{example}[Frattini subgroup]
    Let $G$ be a group. The Frattini subgroup of $G$ is defined as the intersection of its maximal subgroups. The set of subgroups of $G$ form a lower semilattice for the inclusion preorder, and for all subgroups $H,H'$, the meet of $H$ and $H'$ is $H \cap H'$. In this setting, the residual derivative of $G$ is its Frattini subgroup. See Section \textbf{10.4} in \cite{Hall1959} for more details. In this context, our definition of residual rank is reminiscent of the so-called Frattini series, obtained by iterating the operator building the Frattini subgroup. 
\end{example}
 
\begin{example}[Cantor-Bendixson derivative of T1-spaces]\label{example.cb.derivative}
Provided a topological space $X$, its \textit{Cantor-Bendixson derivative}, usually denoted by $X'$, is the topological space which is obtained from $X$ by removing its isolated points. When $X$ is a T1-space, meaning that every singleton of $X$ is closed, it is also the residual derivative of $X$ in the lower semilattice of its closed subspaces. This comes from the fact that a closed subspace is maximal if and only if it is of the form $X \backslash \{x\}$, where $x$ is an isolated point in $X$. Indeed, it is straightforward that a subspace of the form $X \backslash \{x\}$ is closed if and only if $x$ is isolated. When it is, it is clearly maximal. On the other hand, for any $x, y \in X$ such that $x \neq y$, a closed subspace which does not contain $x$ and $y$ is not maximal, because $X \cup \{y\}$ is also closed, different from $X$ and strictly contains $X$.
\end{example}

The 
residual derivative {\blue} generalizes several classical constructions. We conclude this section by discussing {specific aspects of the residual derivative in the context of symbolic dynamics.} 

\begin{remark}[Residual derivative in symbolic dynamics]
For shifts, the 
residual derivative is close to the 
Cantor--Bendixson derivative. 
As a matter of fact, for any shift \(X\) and \(x\in X\) isolated, 
\(X\setminus \mathcal{O}(x)\) is a maximal subelement of \(X\), where \(\mathcal{O}(x)=\{\sigma^{\mathbf{u}}(x):\mathbf{u}\in\mathbb{Z}^d\}\) 
is the orbit of \(x\). 
In particular, every element of \(\mathcal{T}_0(\mathcal{H})\) is perfect, and the residual derivative of $X$ is included in the Cantor-Bendixson derivative of $X$. 
It is important to notice that the two derivatives do not coincide in general. For example, let \(X\) be the disjoint union of a minimal aperiodic $Z$ shift and a full shift $F$. Since both $Z$ and $F$ are perfect, \(X\) is perfect. However, $\mathcal{M}(X) = \{F\}$, thus the residual derivative of $X$ is $F$, that is strictly included in \(X\). This does not contradict Example~\ref{example.cb.derivative}. There, the residual derivative is computed in the lattice of all closed subsets of a \(T_1\)-space, whereas here we considered here the lattice of shift-invariant closed
subsets of a shift. Observe also that in the context of the hyperspace of shifts, the residual derivative, residual rank and  residual core, are topological invariants.
\end{remark}

\section{Residual stratification in dual algebraic coframes\label{section.decomposition}}

Throughout this section, we consider $(L,\le)$ to be a {dual algebraic} coframe {and that $\mathcal{K}(L,\le)$ has a choice function}. We refine the structure induced by the residual derivative by introducing the notion of residue, which captures what is “left out” of a maximal subelement (Section~\ref{section.residues}). This leads to a canonical decomposition of every element into its core and its residues (Section~\ref{section.derivative.residue.decomposition}). Building on this decomposition, we define the boundary poset of an element $x$ as the set of strongly co-irreducible elements above its core $\texttt{c}(x)$ (Section~\ref{section.def.boundary.poset}), and show that every element $z$ such that $\texttt{c}(x) \le z \le x$ can be reconstructed from the core and elements of this boundary. The boundary itself carries a natural stratification reflecting the successive stages of the residual process. As a structural consequence, we prove that the residual derivative is a $\vee$-homomorphism (Section~\ref{section.residual.derivative.vee.homomorphism}), paralleling the behavior of the Cantor–Bendixson derivative. {We also prove that the core $\texttt{c}(x)$ is the join of the elements below $x$ that have no maximal subelement, and that $\texttt{c}$ is also a $\vee$-homomorphism (Section \ref{section.core.operator.properties})}.

The framework developed in this section will serve as the main tool for the analysis of the Cantor–Bendixson layers in the following sections.


Since $L$ is a complete lower semilattice, it has a minimum that we denote by $\varepsilon$. It satisfies $\varepsilon \vee x = x$ and $\varepsilon \wedge x = \varepsilon$ for every $x \in L$, and it is the unique minimal element of $L$. 

\subsection{Residues\label{section.residues}}

Recall that every coframe $L$ carries a Co-Heyting subtraction defined as follows: for $x,z \in L$ with $x \le z$, we denote by $z - x$ the element
\[z - x := \bigwedge \{y \in \mathop{\downarrow}z : x \vee y = z\}.\]
The existence of this meet follows from the fact that $L$ is in particular a complete lower semilattice.

{For the remainder of this section, we fix $F \subset L$ a family which contains $\varepsilon$ and such that $(F,\le)$ is an upper semilattice.}

\begin{definition}
{Let \(x\in L\) and $H$ a family of elements of $L$. An \textbf{\(H\)-residue} of \(x\) is an element of the form \(x-m,\) where \(m\in\mathcal{M}_H(x)\).}
\end{definition}

\begin{lemma}\label{lemma.prop.fund.cosubstr}
    For $x \in L$ and $z \le x$, we have $z \vee (x-z) = x$.
\end{lemma}

\begin{proof}
    {Denote by $S$ 
    the set $\{y \in \mathop{\downarrow}x : z \vee y = x\}$. By the dual infinite distributivity, 
    \[z \vee \bigwedge_{s \in S} s = \bigwedge_{s \in S} (z \vee s) = x.\]
    Finally, by definition of $(z-x)$, we have $z \vee (x-z) = x$.}
\end{proof}

\begin{lemma}\label{lemma.complement.type1}
    For $x \in L$ and $m \in \mathcal{M}_H(x)$, the set $\mathcal{M}_H(x - m)$ has exactly one element and $(x-m)$ has no $H$-outcast. Furthermore, $\mu_H(x-m) \le m$. 
\end{lemma}

\begin{remark}
    {An element $x$ is usually called completely co-irreducible when for any family $(x_i)_{i \in I}$ such that $x = \bigvee_{i \in I} x_i$, there exists $i$ such that $x_i = x$. It can be proved that an element $x$ is completely co-irreducible if and only if $\mathcal{M}(x)$ has a unique element and for all $z \le x$, $z \le \mu(x)$.}
\end{remark}

\begin{proof}
     For $z < (x-m)$, by definition of $(x-m)$, we have $m \vee z < x$. Since $m \in \mathcal{M}_H(x)$ and $m \le m \vee z$, if $z \in H$, which implies $m \vee z \in H$, we must have $m = m \vee z$, which means that $z \le m$. If $\mathcal{M}_H(x - m)$ had at least two elements $z,z'$, by Lemma \ref{lemma.cover.max} we would have $z \vee z' = x-m$. Since $z \le m$ and $z' \le m$, we would have $x-m \le m$, which would imply that $m \vee (x-m) = m$, and this, by Lemma \ref{lemma.prop.fund.cosubstr}, is equal to $x$. Thus $x = m$, which is not true. Therefore $\mathcal{M}_H(x - m)$ has at most one element. Furthermore, there exists some $z \in F$ such that $z < x-m$. Otherwise we would have $(x-m) = \varepsilon$, which is not possible, as it would imply that $m = x$.     
     Thus $\bigvee \{z \in H : z < (x-m)\}$ is well-defined, by Lemma \ref{lemma.upper.complete.ideals}. Since for all $z < (x-m)$, $z \le m$, $\bigvee \{z \in H : z < (x-m)\}$ is distinct from $x-m$. By definition, it is in $\mathcal{M}_H(x-m)$ and it is larger than every $z < (x-m)$ such that $z \in H$, which implies that $(x-m)$ has no $H$-outcast.
\end{proof}

\subsection{\label{section.derivative.residue.decomposition}Core-residues decomposition}

The notion of residue introduced in the previous subsection {corresponds, informally, to}
the part of an element that is discarded when moving to a maximal subelement. In this section we prove that residues completely encode the difference between an element and its residual core. {More precisely, every element admits a canonical decomposition into its core together with the residues produced throughout the construction of the {residual derivative sequence}.}

\begin{lemma}\label{lemma.sum.complements}
    Let $z \in L$ and let $S$ be a subset of $\mathop{\downarrow}z$. Then
    \[z = \left(\bigwedge_{s \in S} s \right)\vee \left(\bigvee_{s \in S} (z-s)\right).\]
\end{lemma}
\begin{proof}
Set \(x := \bigvee_{s \in S}(z-s).\) By dual infinite distributivity, we have: 
\[\left(\bigwedge_{s \in S} s \right)\vee \left(\bigvee_{s \in S} (z-s)\right) = \bigwedge_{s \in S} (x \vee s)\]
For all $s \in S$, we have $x \ge (x-s) $, that is $x \vee s = z$. We then obtain the desired equality.    
\end{proof}

\begin{lemma}\label{lemma.cover.complements}
    For $x \in L$, we have $x = \mu_H(x) \vee \left(\underset{{m \in \mathcal{M}_H(x)}}{\bigvee}(x-m)\right)$.
\end{lemma}

\begin{proof}
{When $\mathcal{M}_H(x)$ is empty, this is straightforward, as in this case $\mu_H(x)= x$. When $\mathcal{M}_H(x)$ is not empty, this is a consequence of Lemma \ref{lemma.sum.complements}.}
\end{proof}

{Lemma~\ref{lemma.cover.complements} shows that an element decomposes into its derivative ad its `first-level' residues. The next two lemmas show how maximal subelements can be reconstructed from these residues. This allows us to iterate the decomposition along the residual derivative sequence.}

\begin{lemma}\label{lemma.formula.maximal}
    For every $x \in L$ and every $m \in \mathcal{M}(x)$, we have: 
    \[m = \mu(x) \vee \left(\underset{\underset{n \neq m}{n \in \mathcal{M}(x)}}{\bigvee} (x-n)\right).\]
\end{lemma}

\begin{proof} For every $m \in \mathcal{M}(x)$, 
set the two following elements: 
 \[m^* := \mu(x) \vee \left(\underset{\underset{n \neq m}{n \in \mathcal{M}(x)}}{\bigvee} (x-n)\right) \quad \text{and} \quad m' := m^* \vee \mu(x-m).\]
 We prove the statement in two steps: first, we prove that for every $m$, we have $m = m'$; and hen that $m=m^*$.
 \textbf{1.} 
    By definition of $\mu(x)$, we have $\mu(x) \le m$. We also have $x-n \le m $ for all $n \in \mathcal{M}(x)$ such that $n \neq m$---this comes from the definition of $x-n$ and the fact that $m \vee n = x$. Because $m$ is a maximal subelement of $x$, we have $ \mu(x-m) \le m$ by Lemma \ref{lemma.complement.type1}. This implies that $m \ge m'$. In order to finish the proof, it suffices to see that $m'$ is a maximal subelement of $x$. Let us consider some element $z \in L$ such that $m' \le z \le x$.
    If we have $z \wedge (x-m) = (x-m)$, then $x-m \le z$, and since $m' \le z$, we have $(x-m) \vee m' \le z$.
    By definition of $m'$, we have 
    $(x-m)\vee m' = m^* \vee (x-m)$. By definition of $m^*$ and Lemma \ref{lemma.cover.complements}, 
    we have $(x-m) \vee m' = x$, which implies that $z = x$. Otherwise, 
    we have $z \wedge (x-m) < (x-m)$. 
    Since $\mu(x-m) \le m' \le z$, 
    we have $z \wedge (x-m) \ge \mu(x-m)$, and thus $z \wedge (x-m) = \mu(x-m)$. By distributivity and Lemma \ref{lemma.cover.complements}, we have $z = z \wedge x = z \wedge (m^* \vee (x-m)) = (z \wedge m^*)\vee (z \wedge (x-m))$. 
    It follows from $m^* \le m' \le z$ that $z \wedge m^* = m^*$. From $z \wedge (x-m) = \mu(x-m)$, we obtain that $z = m'$. This proves that $m'$ is maximal and since $m' \le m$ and $m$ is maximal, we conclude that $m = m'$. \noindent \textbf{2.} As a consequence of the first point,
    \[\left(\underset{{n \in \mathcal{M}(x)}}{\bigvee} \mu(x-n)\right) \le m\]
    for every maximal subelement \(m\) of $x$. Thus: 
    \[\left(\underset{{n \in \mathcal{M}(x)}}{\bigvee} \mu(x-n)\right) \le \mu(x).\]
    In particular, for $n \in \mathcal{M}(x)$, it holds that $\mu(x-n) \le \mu(x)$. By using the first point again, for every $m \in \mathcal{M}(x)$ we have that: 
    \[m = \mu(x) \vee \left(\underset{\underset{n \neq m}{n \in \mathcal{M}(x)}}{\bigvee} (x-n)\right).\]
\end{proof}


\begin{lemma}\label{lemma.surface.tree}
    For $x \in L$ and $m \in \mathcal{M}(\mu(x))$, we have $\mu(x) - m \le \underset{n \in \mathcal{M}(x)}{\bigvee} (x-n)$. Furthermore, when $\mathcal{M}(x)$ is finite, there exists $n$ such that $\mu(x)-m \le x-n$.
\end{lemma}

\begin{proof}
    We proceed by contradiction. Suppose that there exists $m \in \mathcal{M}(\mu(x))$ 
    such that $\mu(x) - m \not \le \underset{n \in \mathcal{M}(x)}{\bigvee} (x-n)$.
    Then 
    \[m':= m \vee \left( \underset{n \in \mathcal{M}(x)}{\bigvee} (x-n)\right)\] is a maximal subelement of $x$.
    Indeed, we have $m' < x$. Otherwise, by distributivity, \[(\mu(x)-m) = (m \wedge (\mu(x)-m)) \vee \left( \left(\underset{n \in \mathcal{M}(x)}{\bigvee} (x-n)\right) \wedge (\mu(x)-m)\right),\]
    and since $(m \wedge (\mu(x)-m)) \le \mu(\mu(x)-m)$, we must have 
    \[\left(\underset{n \in \mathcal{M}(x)}{\bigvee} (x-n)\right) \wedge (\mu(x)-m) = \mu(x)-m.\]
    that is, $\mu(x)-m \le \bigvee_{n \in \mathcal{M}(x)} (x-n)$, which was assumed false. Using Lemma \ref{lemma.cover.complements} on $x$ and $\mu(x)$, for any $z \le x$ that satisfies $z \ge m'$, $z \wedge (\mu(x)-m) = \mu(x)-m$ implies $z =x$ and $z \wedge (\mu(x)-m) < \mu(x)-m$ implies $z= m'$. We have proved that $m'$ is a maximal subelement of $x$. From the definition of $m'$, we have $(x-m') \le m'$, and thus $m' = x$, which is impossible.  
    When $\mathcal{M}(x)$ is finite we thus have, since $L$ is distributive: 
    \[z-m = \underset{n \in \mathcal{M}(x)}{\bigvee} ((x-n) \wedge (z-m)).\]
    Since elements of $\mathop{\downarrow}(z-m)$ different from $z-m$ are smaller than $\mu(z-m)$ (by Lemma \ref{lemma.complement.type1}), there is some $n$ such that $(x-n) \wedge (z-m) = (z-m)$, which means that $(z-m) \le (x-n)$.
\end{proof}

{We can now iterate the decomposition of Lemma~\ref{lemma.cover.complements} through the residual stratification. This yields a decomposition of every element into its residual core and its residues.}


\begin{proposition}\label{lemma.core.complements}
    For every $x \in L$, we have $x = \texttt{c}(x) \vee \left( \underset{n \in \mathcal{M}(x)}{\bigvee} (x-n)\right)$.
\end{proposition}

This core-residues decomposition plays a role analogous to the decomposition of topological spaces into a perfect kernel and isolated layers.

\subsection{\label{section.def.boundary.poset}The boundary poset}

{In this section, we see that the residues of the successive derivatives of an element form a hierarchical structure, each level of this hierarchy corresponding to an element of the residual derivative sequence. We begin by introducing the boundary of an element and the associated boundary poset, which records all residues appearing throughout its residual stratification.
\begin{definition}
For every \(x\in L\), the \textbf{boundary} of \(x\) is the element 
\[\partial x:=\bigvee_{m\in\mathcal{M}(x)}(x-m).\]
\end{definition}

\begin{remark}By Proposition~\ref{lemma.core.complements}, we have \(x=\texttt{c}(x)\vee\partial x\).\end{remark}}

\begin{lemma}\label{lemma.char.outcasts}
    For $x \in L$, the following conditions are equivalent: (i) $x$ has an outcast; (ii) $\texttt{c}(x) \not \le \partial x$; (iii) $\partial x < x$. When these conditions are satisfied, the outcasts of $x$ are exactly the elements of $\mathop{\uparrow} (\partial x) \setminus \{x\}$.
\end{lemma}

\begin{proof}
    $(ii) \Rightarrow (iii)$. If $\texttt{c}(x) \not \le \partial x$, since $x = \texttt{c}(x) \vee \partial x$, we must have $\partial x < x$. 
    $(iii) \Rightarrow (i)$. If $\partial x < x$, $\partial x$ is an outcast: if there was any $m \in \mathcal{M}(x)$ such that $\partial x \le m$, we would have $(x-m) \le m$  {and since $x=(x-m)\vee m$, we obtain $m = x$, a contradiction.}
    $(i) \Rightarrow (iii)$ If $\partial x = x$, then for all $z < x$, 
    there must be some $m \in \mathcal{M}(x)$ such that $z \wedge (x-m) < (x-m)$ (otherwise we would have $z = x$), which implies that $z \le m$. This means that $x$ has no outcast. $(iii) \Rightarrow (ii)$. If $\texttt{c}(x) \le \partial x$, since $x =\texttt{c}(x) \vee \partial x$, we get $x = \partial x$.  
    Assume that these conditions are satisfied. For all $z < x$ such that 
    $z \ge \partial x$, $z$ is an outcast because $\partial x$ is an outcast. For all $z < x$ such that $z \not \ge \partial x$, there exists $m$ such that $z \wedge (x-m) < (x-m)$, which means that $z \le m$, and $z$ is not an outcast of $x$.
\end{proof}

\begin{remark}
    It follows from Lemma \ref{lemma.char.outcasts} that any $x \in L$ with $\texttt{c}(x) = \varepsilon$ has no outcast.
\end{remark}

\begin{definition}
    For $x \in L$  and an ordinal $\alpha$, we define $\alpha$-th \textbf{stratum} of $x$ to be the set
    \[\texttt{s}_{\alpha}(x) := \{x^{(\alpha)} - m : m \in \mathcal{M}(x^{(\alpha)})\}.\]
    We also define the \textbf{boundary poset} of $x$ as the set $\delta(x) := \bigcup_{\alpha} \texttt{s}_{\alpha}(x)$.
\end{definition}

\begin{remark}
     Observe that by definition we have that $\texttt{s}_{\alpha}(x) = \emptyset$ whenever $\alpha \geq \texttt{r}(x)$. Observe as well that for any $\alpha$, $\bigvee \texttt{s}_{\alpha}(x) = \partial x^{(\alpha)}$.
\end{remark}

{
The boundary poset consists in all residues that arise during the construction of the residual derivative sequence of \(x\), grouped according to the ordinal at which they appear. We first prove that the strata make it a transfinite ranked poset.}


For  $x \in L$, the residues of $x$ form a ``minimal cover'' of $\delta x$, in the following sense.

\begin{lemma}\label{lemma.noncomparable}
    For all $x \in L$ and $m \in \mathcal{M}(x)$: \[(x-m) \not \le \left(\bigvee_{\underset{n \neq m}{n \in \mathcal{M}(x)}} (x-n) \right).\]
\end{lemma}

\begin{proof}
    This derives from Lemma \ref{lemma.formula.maximal}: if this was false, we would have $m = x$, which is not possible.
\end{proof}




\begin{definition}
    We call transfinite ranked poset a triple $(X,\le,\rho)$ 
    such that $(X,\le)$ is a poset and $\rho : X \rightarrow \lambda$, where $\lambda$ is an ordinal, such that for all $x,y \in X$, if $x < y$, $\rho(x) < \rho(y)$.
\end{definition}

\begin{notation}
    For $x \in L$, we denote by $\rho_x$ the function $\delta(x) \rightarrow \texttt{r}(x)$ such that 
    for every $s \in \delta(x)$, $\rho_x(s)$ is the unique $\alpha < \texttt{r}(x)$ such that $s \in \texttt{s}_{\alpha}(x)$.
\end{notation}

\begin{proposition}
    The triplet $(\delta(x),\ge,\rho_x)$ is a transfinite ranked poset.
\end{proposition}

\begin{proof}
    We prove that for $s,t \in \delta(x)$, if $s < t$
    then $\rho_x(s) > \rho_x(t)$. 
    {First, observe that if \(\rho_x(s)=\rho_x(t)\), then \(s\) and \(t\) are in \(\texttt{s}_{\rho_x(s)}(x)\). Therefore, there exist \(m,n\in\mathcal{M}(x^{(\rho_x(s))})\) with \(m\neq n\) and such that \(s=x^{(\rho_x(s))}-m\) and \(t=x^{(\rho_x(s))}-n\). {Since \(s< t\), we have
    \[x^{(\rho_x(s))}-m\leq x^{(\rho_x(s))}-n,\]
    which contradicts Lemma \ref{lemma.noncomparable}.}
    Hence, \(\rho_x(s)\neq\rho_x(t)\).}
    
    Now assume, ad absurdum, that $\rho_x(s)<\rho_x(t)$, and set
    $\alpha:=\rho_x(s)$. Since $t\in \texttt{s}_{\rho_x(t)}(x)$ and
    $\rho_x(t)>\alpha$, we have
    \[
        t\le x^{(\alpha+1)}=\mu(x^{(\alpha)}).
    \]
    Because $s<t$, this implies
    \[
        s\le \mu(x^{(\alpha)}).
    \]
    But $s\in \texttt{s}_\alpha(x)$, so $s=x^{(\alpha)}-m$ for some
    $m\in\mathcal M(x^{(\alpha)})$. By Lemma~\ref{lemma.prop.fund.cosubstr},
    \[
        m\vee s=x^{(\alpha)}.
    \]
    If $s\le \mu(x^{(\alpha)})$, then, since $\mu(x^{(\alpha)})\le m$, we get $s\le m$. Hence
    \[
        x^{(\alpha)}=m\vee s=m,
    \]
    contradicting $m<x^{(\alpha)}$. Therefore $\rho_x(s)<\rho_x(t)$ is impossible,
    and thus
    \[
        s<t \Longrightarrow \rho_x(s)>\rho_x(t).
    \]
\end{proof}

{Recall that a ranked poset $(X,\le,\rho)$ is said to be graded when 
for all $x$ and $y$ in $X$, if $y$ covers $x$ then $\rho(y) = \rho(x) +1$.}

\begin{question}
    Is $(\delta(x),\ge,\rho_x)$ a graded poset?
\end{question}

{This would mean that for all \(s \in \delta(x)\) and \(t \in \mathcal{M}(s)\), there exists $\alpha$ such that $s$ and $t$ belong to consecutive strata \(\texttt{s}_{\alpha}(x)\) and \(\texttt{s}_{\alpha+1}(x)\). The proposition above only establishes strict monotonicity of \(\rho_x\) along the order, which is weaker: some ordinals may be skipped. Whether the graded property holds seems to depend on finer properties of the poset \((L,\le)\).}

{The ranked poset structure on \(\delta(x)\) derives directly from its definition from the strata, which are themselves defined in terms of the residual derivative sequence 
\((x^{(\alpha)})_{\alpha<\texttt{r}(x)}\). It appears that this $\delta(x)$
can be expressed without reference 
to the residual derivative sequence.}

\begin{notation}
We denote by $\mathcal{I}(L,\le)$ the set $\{x \in L : |\mathcal{M}(x)|=1, z < x \Rightarrow z \le \mu(x)\}$, {that is, the set of completely co-irreducible elements in $\mathop{\downarrow}x$}, and by ${\delta}^+(x)$, \(x\in L\), the set of $s\in \mathcal{I}(L,\le)$ such that $s \le x$ and $s \not \le \texttt{c}(x)$.
\end{notation}

{We prove below that $\delta(x) = \delta^+(x)$, which is a key ingredient to prove that every element above $\texttt{c}(x)$ and below $x$ can be written as the join of $\texttt{c}(x)$ with elements of $\delta(x)$. Before this, we need the following lemmas.}


\begin{lemma}\label{lemma.complement.uniqueness}
    For $x \in L$ and $s \in \texttt{s}_0(x)$, there is a unique $m \in \mathcal{M}(x)$ such that $s \vee m = x$.
\end{lemma}

\begin{proof}
Fix $x \in L$, $s \in \texttt{s}_0(x)$, and suppose that $m_1, m_2 \in \mathcal{M}(x)$ are such that $s \vee m_1 = x$ and $s \vee m_2 = x$. We prove that $m_1 = m_2$. We have $s \wedge m_1 \le \mu(s) \le m_2$ and $s \wedge m_2 \le \mu(s) \le m_1$. By distributivity, $m_2 = m_2 \wedge (s \vee m_1) = (m_2 \wedge s) \vee (m_2 \wedge m_1) \le m_1$. An analogous argument shows that $m_1 \le m_2$.
\end{proof}


\begin{lemma}\label{lemma.completeness}
     For every $x$, we have 
    \[\texttt{s}_0(x) = \left\{s \in {\delta}^+(x) : s \not \le \bigvee_{t \in \delta^+(x) \setminus \{s\}}t\right\},\]
    and for every $s \in \texttt{s}_0(s)$, $m=\texttt{c}(x) \vee \bigvee_{t \in \delta^+(x) \setminus \{s\}}t$ is the unique element in $\mathcal{M}(x)$ such that $s \vee m = x$.
\end{lemma}

\begin{proof}
    \textbf{1.} We prove first the inclusion $\subset$. Consider $s \in \texttt{s}_0(x)$, and $m \in \mathcal{M}(x)$ such that $s = (x-m)$. We have $s \vee m = x$.
    For all $t \in \delta^+(x) \setminus \{s\}$, we have $s \wedge t \le \mu(s)$, which means, 
    since $t = (t \wedge s) \vee (t \wedge m)$, 
    that $t \le \mu(s) \vee m \le m$. This implies that: 
    \[\bigvee_{t \in \delta^+(x) \setminus \{s\}}t \le m, \qquad \text{and thus }\qquad s \not \le \bigvee_{t \in \delta^+(x) \setminus \{s\}}t.\]
    {\textbf{2.} We prove now the inclusion $\supset$. Fix $s \in \delta^+(x)$ such that \[s \not \le \bigvee_{t \in \delta^+(x) \setminus \{s\}}t.\]
    Then $m:= \texttt{c}(x) \vee \bigvee_{t \in \delta^+(x) \setminus \{s\}}t$ is an element of $\mathcal{M}(x)$. Indeed, we know by Lemma \ref{lemma.core.complements} and point 1 that $x = \texttt{c}(x) \vee \bigvee_{t \in \delta^+(x)} t$. 
    Considering any $z$ such that $m \le z \le x$, if $s \le z$, clearly $z = x$. Otherwise $z \wedge s \le \mu(s)$. We have $\mu(s) \le \texttt{c}(s) \vee \bigvee_{t \in \texttt{s}_0(s)} t \le \texttt{c}(x) \vee \bigvee_{t \in \delta^+(x) \setminus \{s\}}t$.
    This means that $z \wedge s \le m$.
    Thus, since $z \wedge x = (z \wedge s) \vee (z \wedge m) = (z\wedge s) \vee m \le m \vee m \le m$.
    We conclude that $z = m$. 
    {We have proved that $m \vee s = x$. Uniqueness derives from Lemma \ref{lemma.complement.uniqueness}}.
    }
\end{proof}

As a consequence:

\begin{lemma}
    For $x\in L$, we have $\delta(x) = \delta^+(x)$.
\end{lemma}

\begin{proof}
    {From the definitions, we directly have $\delta(x)\subset \delta^+(x)$. We need to prove that $\delta(x)\supset \delta^+(x)$. Assume that $\delta^+(x) \backslash \delta(x)$ is not empty, and denote by $s$ one of its elements. By transfinite induction, for $\alpha < \texttt{r}(x)$, we have that $s \le x^{(\alpha)}$. Indeed, assuming this is true for $\alpha$, since $s \in \mathcal{I}(L,\le)$, we have $s \in \texttt{s}_{\alpha}(x)$ or $s \le \mu(x^{(\alpha)})$. Since by hypothesis $s \notin \delta(x)$, we have $s \le \mu(x^{(\alpha)}) = x^{(\alpha+1)}$.
    This means that $s \le \texttt{c}(x)$, which is not possible, by definition of $\delta^+(x)$. If it holds for every $\alpha < \lambda$, where $\lambda$ is a limit ordinal, then it also holds for $\lambda$.
    }
\end{proof}

{Having established an intrinsic characterisation of \(\delta(x)\) that does not depend on the residual sequence, we now turn to the analysis of the behavior of the boundary with respect to lattice operations. The following lemma shows that the boundary set is additive under join.}

\begin{lemma}\label{prelemma.subadditivity.type}
{For $x, z \in L$, we have $\delta^+(x \vee z) = \delta^+(x) \cup \delta^+(z)$.}
\end{lemma}

\begin{proof}
    Indeed, the inclusion $\delta^+(x) \cup \delta^+(z) \subset \delta^+(x \vee z)$ is immediate from the definition. In order to see that $\delta^+(x \vee z) \subset \delta^+(x) \cup \delta^+(z)$, it is sufficient to see that for all $s \le x \vee z$ such that $s \in \mathcal{I}(L,\le)$, we have $s \le x$ or $s \le z$. We prove this. For all $s \in \mathcal{I}(L,\le)$ such that $s \le x \vee z$, we have $s = s \wedge (x \vee z)$. By distributivity, 
    $s = (s \wedge x) \vee (s \wedge z)$. Since $s \in \mathcal{I}(L,\le)$, this implies that $s \wedge x = s$ or $s \wedge z = s$, meaning that $s \le x$ or $s \le z$. From Lemma \ref{lemma.completeness}, 
    we obtain that 
    $\texttt{s}_0(x\vee z) \subset \texttt{s}_0(x) \cup \texttt{s}_0(z)$. The statement follows.
\end{proof}

{As a direct consequence:}

\begin{lemma}\label{lemma.subadditivity.type}
For $x, z \in L$, $|\mathcal{M}(x \vee z)| \le |\mathcal{M}(x)| + |\mathcal{M}(z)|$. In particular $\bigcup_{n < \infty} \mathcal{T}_n(L,\le)$ is an upper semilattice.
\end{lemma}




\begin{lemma}\label{lemma.minmax}
Let \(S\) be any set. Let \((x_s)_{s \in S}\) and \((z_s)_{s \in S}\) be families in \(L\). We have 
\[\bigwedge_{s \in S} (x_s \vee z_s) \leq(\bigvee_{s \in S} x_s)\vee(\bigwedge_{s \in S} z_s).\]
\end{lemma}

\begin{proof}
    {Set $x = \bigvee_s x_s$. We prove that $\bigwedge_{s \in S} (x_s \vee z_s) \le x \vee \left( \bigwedge_{s \in S} z_s\right)$. By dual infinite distributivity,
    we have \[x \vee \left( \bigwedge_s x_s\right) = \bigwedge_s (x \vee z_s).\]
    Since we have $x \vee z_s \ge x_s \vee z_s$, 
    we have 
    \[\bigwedge_{s \in S} (x_s \vee z_s) \le x \vee \left(\bigwedge_{s} z_s\right).\]}
\end{proof}

{This lemma is a key tool for handling limits of joins of a non-decreasing and a non-increasing sequences in the transfinite induction in the proof of Lemma~\ref{lemma.subelement.decomposition}. 
}

\begin{lemma}\label{lemma.subelement.decomposition}
    For every $x \in L$ and $ z \in \mathop{\downarrow}x$, we have 
    \[z = (z \wedge \texttt{c}(x)) \vee \bigvee_{\underset{s \le z}{s \in \delta(x)}} s.\]
\end{lemma}

\begin{proof}
    \textbf{1.} 
    {We first prove the statement when \(z\ge \texttt{c}(x)\). For every \(\alpha<\texttt r(x)\), set

\[u_{\alpha} := x^{(\alpha)}\vee \left(\bigvee_{\beta < \alpha} \bigvee_{\underset{s \le z}{s \in \texttt{s}_{\beta}(x)}}s\right)\]

We prove by transfinite induction on \(\alpha\) that \(z\leq u_\alpha\). We have $u_0 = x$, therefore $z \le u_{0}$ is straightforward. Suppose now that the statement holds for every \(\beta<\alpha\). Assume that \(\alpha\) is a limit ordinal. We apply Lemma \ref{lemma.minmax}, where $S = \texttt{r}(x)$, and for all $\beta \in \texttt{r}(x)$,
\[x_{\beta} := \bigvee_{\lambda < \alpha} \bigvee_{\underset{s \le z}{s \in \texttt{s}_{\lambda}(x)}}s \qquad \text{and} \qquad z_{\beta} := x^{(\beta)}.\]
We have $\bigvee_{\beta < \alpha} x_{\beta} = x_{\alpha}$ and 
$\bigwedge_{\beta < \alpha} z_{\beta} = z_{\alpha}$.
We thus have \(u_\alpha \ge \bigwedge_{\beta<\alpha}u_\beta\). Since \(z\leq u_\beta\) for every \(\beta<\alpha\), we obtain \(z\le u_\alpha\).} 
Otherwise we set $\lambda$ such that $\lambda + 1 = \alpha$. By distributivity, for every finite subset {$P \subset \texttt{s}_{\lambda}(x)$}:
{
\[
\begin{aligned}
z = z \wedge u_{\lambda}
  ={}&(z \wedge \mu(x^{(\lambda)}))
   \vee \left(z \wedge \bigvee_{\beta < \lambda}
        \bigvee_{\substack{s \in \texttt{s}_{\beta}(x)\\ s \le z}} s\right) \\
   &\vee \left(\bigvee_{\underset{s \le z}{s \in P}} z \wedge s\right) \vee \left(z \wedge\bigvee_{\underset{s \not \le z}{s \in P}} s\right)
   \vee \left(z \wedge \bigvee_{s \in \texttt{s}_{\lambda}(x)\setminus P} s\right).
\end{aligned}
\]

For $s \in P$, if $z \wedge s < s$ then $z \wedge s \le \mu(x^{(\lambda)})$, and thus 
$z \wedge s \le z \wedge \mu(x^{(\lambda)})$. We can thus rewrite: 

\[
\begin{aligned}
z \wedge u_{\lambda}
  ={}&(z \wedge \mu(x^{(\lambda)}))
   \vee \left(z \wedge \bigvee_{\beta < \lambda}
        \bigvee_{\substack{s \in \texttt{s}_{\beta}(x)\\ s \le z}} s\right) \\
   &\vee \left(\bigvee_{\underset{s \le z}{s \in P}} z \wedge s\right) \vee \left(z \wedge \bigvee_{s \in \texttt{s}_{\lambda}(x)\setminus P} s\right).
\end{aligned}
\]

We apply again Lemma \ref{lemma.minmax}, with $S$ equal to the set of finite subsets of $\texttt{s}_{\lambda}(x)$, and for all $P \in S$, 
\[x_P := (z \wedge \mu(x^{(\lambda)}))
   \vee \left(z \wedge \bigvee_{\beta < \lambda}
        \bigvee_{\substack{s \in \texttt{s}_{\beta}(x)\\ s \le z}} s\right) \vee \left(\bigvee_{\underset{s \le z}{s \in P}} z \wedge s\right),\]
and 
\[z_P := \left(z \wedge \bigvee_{s \in \texttt{s}_{\lambda}(x)\setminus P} s\right)\]

By application of the lemma, 
we have $z \wedge u_{\lambda} \le \left(\bigvee_{P \in S} x_P\right) \vee \left(\bigwedge_{P \in S} z_P\right)$. We have 
\[\bigwedge_{P \in S} z_P \le z \wedge \bigwedge_{\underset{P finite}{P \subset \mathcal{M}(x^{(\lambda)})}} \bigwedge_{m \in \mathcal{M}(x^{(\lambda)})\setminus P} m \le z \wedge \mu(x^{(\lambda)})\]
We thus have also 
\[z \le z \wedge \left(\mu(x^{(\lambda)}) \vee  \left(\bigvee_{\beta \le \lambda}
        \bigvee_{\substack{s \in \texttt{s}_{\beta}(x)\\ s \le z}} s\right)\right) = z \wedge u_{\alpha}.\]
In particular, $z \le u_{\alpha}$.
}


{
\textbf{2.} We apply the first point on $z \vee \texttt{c}(x)$, and get: 
\[z \vee \texttt{c}(x) = \texttt{c}(x) \vee \bigvee_{\underset{s \le z}{s \in \delta(x)}} s.\]
Then by distributivity 
\[z = z \wedge (z \vee \texttt{c}(x)) = (z \wedge \texttt{c}(x)) \vee \left(z \wedge  \bigvee_{\underset{s \le z}{s \in \delta(x)}} s\right).\]

Since the last term is below $z$, 
we get 
\[z = (z \wedge \texttt{c}(x)) \vee \bigvee_{\underset{s \le z}{s \in \delta(x)}} s.\]
}
\end{proof}

\begin{lemma}
    For every $x \in L$ and every finite subset $P \subset \delta(x)$, 
    there exists a finite sequence $x_0 , \ldots , x_k$ 
    such that $x_0 = x$,
    \[x_k = \texttt{c}(x) \vee \bigvee_{s \in \delta(x) \setminus P} s.\]
    Moreover, $x_{i+1} \in \mathcal{M}(x_i)$, for any $i<k$.
\end{lemma}

\begin{corollary}\label{corollary.submax.dominant}
    If $x$ is dually compact and $P \subset \delta(x)$ is finite, then $\bigvee_{s \in \delta(x) \setminus P} s$ is also dually compact.
\end{corollary}

\begin{proof}
If $\delta(x)$ is empty, the statement is trivial. Assume that it is not empty. There exists a sequence $s_1, \ldots, s_k$ of elements of $\delta^+(x)$ such that for every $l < k$, we have: 
\[s_l \not \le \bigvee_{s \in \delta^+(x) \setminus \{s_i : i \le l\}} s.\]
For $l \le k$, we set 
\[x_l := \texttt{c}(x) \vee \bigvee_{s \in \delta^+(x) \setminus \{s_i : i \le l\}} s.\]
We prove that for $l < k$, $x_{l+1} \in \mathcal{M}(x_l)$, which ends the proof. For this, it is sufficient to see that $\mu(s_l) \le x_l$. {Indeed, we have $\texttt{s}_0(\mu(s_l)) \cap \{s_i : i \le l\} = \emptyset$. Since $\delta(x) = \delta^+(x)$, this implies that $\partial \mu(s_l) \le x_l$. Furthermore, since $s_l \le x$, $\texttt{c}(\mu(s_l)) \le \texttt{c}(x) \le x_l$. We thus have $\mu(s_l) = \texttt{c}(s_l) \vee \partial \mu(s_l) \le x_l$.}
\end{proof}


\begin{remark}
    The results in this section become easier to prove when the lattice is assumed to be completely distributive. However, the family of lattices that we consider in this text includes important examples which do not satisfy this hypothesis.
\end{remark}


\subsection{\label{section.residual.derivative.vee.homomorphism}The residual derivative is a $\vee$-homomorphism}

{ 
We prove here that the residual derivative \(\mu\) is a \(\vee\)-homomorphism, that is, the operator \(\mu\) preserves joins. 
This property is particularly significant because it shows that the residual derivative is not merely defined in terms of the order structure, but is compatible with the algebraic operations of the coframe. In this sense, \(\mu\) behaves as a canonical lattice-theoretic operator, analogous to familiar closure and interior operators arising in topology and order theory.}

\begin{proposition}\label{proposition.increase.derivative}
    For $x \in L$ and $z \in \mathop{\downarrow}x$, we have $\mu(z) \le \mu(x)$. 
    \end{proposition}

\begin{proof}
     {We know that $x = \mu(x) \vee \left(\bigvee_{s \in \texttt{s}_0(x)} s\right)$.}
     This implies that 
    $\mu(z) = (\mu(x) \wedge \mu(z)) \vee \left(\bigvee_{s \in \texttt{s}_0(x)} s\wedge \mu(z)\right)$. For all $s \in \texttt{s}_0(x)$, {if $s \le z$, then $s \in \texttt{s}_0(z)$ and thus $s \not \le \mu(z)$. If $s \not \le z$, $s \not \le \mu(z)$. This condition implies that $\mu(z) \wedge s < s$ and therefore $\mu(z) \wedge s \le \mu(s)$. Since $\mu(s) \le \mu(x)$, for all $s \in \texttt{s}(x)$, $s \wedge \mu(z) \le \mu(x)$, and since $s \wedge \mu(z) \le \mu(z)$, we also have $s \wedge \mu(z) \le \mu(x) \wedge \mu(z)$. We have $\mu(z) \le \mu(x) \wedge \mu(z)$, which implies that $\mu(z) \le \mu(x)$.}
\end{proof}

{The monotonicity of \(\mu\) established above immediately gives one half of the homomorphism identity: for all $x,z \in L$, since \(x,z\leq x\vee z\), we have \(\mu(x)\leq\mu(x\vee z)\) and \(\mu(z)\leq\mu(x\vee z)\), and thus \(\mu(x)\vee\mu(z)\leq\mu(x\vee z)\). The following theorem states that the equality holds.}
\begin{theorem}
{The residual derivative is a \(\vee\)-homomorphism, that is,} for any $x,z \in L$, $\mu(x \vee z) = \mu(x) \vee \mu(z)$. 
\end{theorem}

\begin{proof}
    We have $\mu(x) \vee \mu(z) \le \mu(x\vee z)$ as a consequence of Proposition \ref{proposition.increase.derivative}.
    We prove that $\mu(x) \vee \mu(z) \ge \mu(x\vee z)$. From Lemma \ref{lemma.cover.complements}, we have 
    \[x \vee z = \mu(x)\vee \mu(z) \vee \left(\underset{{m \in \mathcal{M}(x)}}{\bigvee}(x-m)\right) \vee \left(\underset{{m \in \mathcal{M}(z)}}{\bigvee}(z-m)\right).\]
    We can rewrite this as: 
    \[x \vee z = \mu(x)\vee \mu(z) \vee \left(\underset{\underset{x-m \not \le \mu(z)}{m \in \mathcal{M}(x)}}{\bigvee}(x-m)\right) \vee \left(\underset{\underset{z-m \not \le \mu(x)}{m \in \mathcal{M}(z)}}{\bigvee}(z-m)\right).\]
    It is thus sufficient to see that for all $m' \in \mathcal{M}(x)$ such that $x-m' \not \le \mu(z)$, 
    \[\mu(x)\vee \mu(z) \vee \left(\underset{\underset{x-m \not \le \mu(z)}{m \in \mathcal{M}(x) \setminus \{m'\}}}{\bigvee}(x-m)\right) \vee \left(\underset{\underset{z-m \not \le \mu(x)}{m \in \mathcal{M}(z)}}{\bigvee}(z-m)\right)\]
    and for all $m' \in \mathcal{M}(z)$ such that $z-m' \not \le \mu(x)$, 
    \[\mu(x)\vee \mu(z) \vee \left(\underset{\underset{x-m \not \le \mu(z)}{m \in \mathcal{M}(x)}}{\bigvee}(x-m)\right) \vee \left(\underset{\underset{z-m \not \le \mu(x)}{m \in \mathcal{M}(z)\setminus \{m'\}}}{\bigvee}(z-m)\right)\]
    are maximal subelements of $x \vee z$. This uses similar arguments as the proof of Lemma \ref{lemma.formula.maximal}.
\end{proof}


\subsection{Properties of the core operator\label{section.core.operator.properties}}

{In this section, we prove some properties of the core operator $\texttt{c}$. In particular, for $x$, $\texttt{c}(x)$ is the join of all elements below $x$ that have no maximal subelement (Lemma \ref{lemma.core.union}). We also prove that $\texttt{c}$ is a $\vee$-homomorphism (Lemma \ref{lemma.lattice.op}).}

{We start by showing that the join of all elements below $x$ that have no maximal subelements is well-defined.}


\begin{lemma}\label{lemma.t0.uppersemilattice}
    {For every $z \in L$,} the pair $(\mathcal{T}_0 (L) {\cap (\mathop{\downarrow}z)}, \le)$ is a complete upper semilattice. 
\end{lemma}

\begin{proof}
    It is sufficient to see that for all $S \subset \mathcal{T}_0(L,\le){ \cap (\mathop{\downarrow}z)}$, $\bigvee S \in \mathcal{T}_0(L,\le)$. Fix $S \subset \mathcal{T}_0(L,\le){\cap (\mathop{\downarrow}z)}$, and $x < \bigvee S$. 
    We can assume that for all $S' \subsetneq S$, $\bigvee S' < \bigvee S$. 
    {We have $x \ge \bigvee_{s \in S}(x \wedge s)$.}
    If $(x \wedge s) = s$ for all $s \in S$, then
    {$x \ge \bigvee S$, which is a contradiction.}
    Therefore, there is some $s \in S$ such that 
    $x \wedge s < s$. Since $s \in \mathcal{T}_0(L,\le)$, there exists $z$ such that $x \wedge s < z < s$. Then $x < x \vee z < \bigvee S$. Indeed, if $x = x \vee z$, then $z \le x$ and thus $z = z \wedge s \le x \wedge s$, which is false. If $x \vee z = \bigvee S$, we have $(x \vee z) \wedge s = s$, and thus $z = (x \wedge s) \vee (z \wedge s) = s$, which is also false.
\end{proof}

\begin{lemma}\label{lemma.core.union}
For $x \in L$, we have: 
\[\texttt{c}(x) = \bigvee_{z \in \mathcal{T}_0(L,\le) \cap (\mathop{\downarrow} x)} z.\]
\end{lemma}
 
\begin{proof}
Since $\texttt{c}(x) \in \mathcal{T}_0(L,\le) \cap (\mathop{\downarrow} x)$, we have $\texttt{c}(x) \le \bigvee_{z \in \mathcal{T}_0(L,\le) \cap (\mathop{\downarrow} x)} z$. 
We have to prove that for all $z \in \mathcal{T}_0(L,\le) \cap (\mathop{\downarrow} x)$, $z \le \texttt{c}(x)$. We do this using transfinite induction to prove that $P(\alpha)$: for any $\alpha < \texttt{r}(x)$, we have $z \le x^{(\alpha)}$ for all $z \in \mathcal{T}_0(x)$. The base case  $P(0)$ is trivial. Fix $\alpha < \texttt{r}(X)$ such that $P(\alpha)$ holds and $\alpha+1 < \texttt{r}(X)$. As a consequence of $P(\alpha)$, in order to prove $P(\alpha +1)$ it is sufficient to prove that every maximal element of $(\mathop{\downarrow} x^{(\alpha)}) \setminus \{x^{(\alpha)}\}$ contains every element of $\mathcal{T}_0(L,\le) \cap (\mathop{\downarrow}x^{(\alpha)})$. Assume that it is not the case, and fix $z \in \mathcal{T}_0(x^{(\alpha)})$ and $m$ a maximal subelement of $(\mathop{\downarrow} x^{(\alpha)}) \setminus \{x^{(\alpha)}\}$ such that $z \not \le m$. In particular, 
$z \wedge m < z$. Since $z \in \mathcal{T}_0(x^{(\alpha)})$, there exists $z' < z$ such that $z \wedge m < z'$.
This implies that $m < m \vee z' < x^{(\alpha)}$, which is impossible, since $m$ is a maximal subelement of $x^{(\alpha)}$. Indeed, if $m = m \vee z'$, then $z' \le m$ and thus, since $z' \le z$, we have $z' \le z \wedge m$, which is not true; if $x^{(\alpha)} = m \vee z'$, then $z = (m \wedge z) \vee (z \wedge z') \le z'$, which is not true. 
We have now proved $P(\alpha +1)$. Consider a limit ordinal $\lambda < \texttt{r}(X)$, and assume that $P(\beta)$ holds for all $\beta < \lambda$. Then $P(\lambda)$ holds. This comes directly from the fact that $x^{(\lambda)} = \bigwedge_{\beta < \lambda} x^{(\beta)}$. 
\end{proof}

\begin{remark}
   Lemma \ref{lemma.core.union} implies that for all $x \notin \mathcal{T}_0(L,\le)$, 
$\texttt{c}(x) = \mu_{\mathcal{T}_0(L,\le)}(x)$.
\end{remark}

{
We prove that \(\texttt{c}\) interacts cleanly with the lattice operations. The key ingredient in the proof of Lemma \ref{lemma.lattice.op} is the following result.}

\begin{lemma}\label{lemma.core.decomposition}
    For all $x,y,z \in L$ such that  $y \le x \vee z$ and $y \in \mathcal{T}_0(L,\le)$, we have $y = \texttt{c}(x \wedge y) \vee \texttt{c}(z \wedge y)$.
\end{lemma}
\begin{proof}
    By distributivity of $L$, $y = (x \wedge y) \vee (z \wedge y)$. If $x \wedge y$ has some maximal subelement $m$, we have $(x \wedge y) - m \le z \wedge y$. Otherwise, $(z \wedge y) \vee m$ would be a maximal subelement of $y$, which is not possible since $y \in \mathcal{T}_0(L,\le)$. Repeating this reasoning, we obtain $y = \mu(x \wedge y) \vee (z \wedge y)$. Using transfinite induction, we then get 
    $y = \texttt{c}(x \wedge y) \vee (z \wedge y)$. Symmetrically, we obtain $y = \texttt{c}(x \wedge y) \vee \texttt{c}(z \wedge y)$.
\end{proof}

{
Combining Lemmas \ref{lemma.core.union} and \ref{lemma.core.decomposition} gives the distributivity of \(\texttt{c}\) over joins.}

\begin{lemma}\label{lemma.lattice.op}
    For all $x, z \in L$, we have $\texttt{c}(x \vee z) = \texttt{c}(x) \vee \texttt{c}(z)$. In particular, when $x \le z$, we have $\texttt{c}(x) \le \texttt{c}(z)$.
\end{lemma}

\begin{proof}
    Lemma \ref{lemma.core.union} implies that $\texttt{c}(x) \vee \texttt{c}(z) \le \texttt{c}(x \vee z)$. From Lemma \ref{lemma.core.decomposition}, we have $\texttt{c}(x \vee z) \le \texttt{c}(x) \vee \texttt{c}(z)$.
\end{proof}

\begin{remark}
    The results in this section imply that the set of elements of $L$ whose core is equal to $\epsilon$ is a sublattice of $(L,\le)$. 
\end{remark}

\section{\label{section.lattice.proofs.1}Characterization of $\mathcal{S}_0(L,\tau) \setminus \mathcal{S}_1(L,\tau)$}

In this section, we focus on the first level of the Cantor--Bendixson hierarchy. Throughout this section, \((L,\leq)\) denotes a dual algebraic coframe and \(\tau\) an order-compatible topology on \(L\).

Our goal is to characterize the isolated points of \((L,\tau)\) purely in terms of the residual structure developed in the previous sections. More precisely, we shall prove (Theorem \ref{thm.characterization.isolated}) that an element is isolated if and only if it satisfies {three} 
order-theoretic conditions: it is dually compact, has no outcast and possesses only finitely many maximal subelements---extending Theorem A from \cite{GN25}. Thus the first Cantor-Bendixson level admits a description that makes no explicit reference to the topology.

{
We first prove that the points satisfying the {three} conditions above are isolated (Proposition \ref{prop.no.outcast}), and then prove the converse implication with a series of propositions.
}

\begin{remark}
Notice that there is one subtle difference between the framework introduced in the present text and the one of \cite{GN25}: here the empty set is considered to be a shift, which implies that minimal shifts have a unique maximal subsystem, whereas in \cite{GN25} they are considered to have no subsystem and therefore no maximal subsystem. This convention turns out to be more natural in the lattice-theoretic setting and simplifies several definitions and constructions. Observe also that $(\mathcal{H}^d,\tau_{\mathcal{H}}^d)$ is a T1 space, so its Cantor-Bendixson structure coincides with its residual structure.
\end{remark}


\begin{proposition}\label{prop.no.outcast}
    If $x \in L$ is dually compact, $x$ has no outcast and $\mathcal{M}(x)$ is finite, then $x$ is isolated. 
\end{proposition}

\begin{proof}
    {Since $x$ is dually compact, $\mathop{\downarrow}x$ is open. Thus if $x$ were not isolated, there would exist a net $\boldsymbol{x} : D \rightarrow L$ converging to $x$ and such that $\boldsymbol{x}_d < x$ for all $d \in D$. Since $x$ has no outcast and $\mathcal{M}(x)$ is finite, without loss of generality we can assume that there exists $m \in \mathcal{M}(x)$ such that $\boldsymbol{x}_d \le m$ for all $d \in D$. We thus have 
    $\boldsymbol{x}_d \vee m = m$. 
    By continuity of $\vee$, we get $x \vee m = m$, which is impossible since $x \vee m = x$ and $m < x$.}
\end{proof}

{
We prove the converse implication. We begin by proving that isolated points must be dually compact and have finitely many maximal subelements.
}
\begin{proposition}\label{isolated.finite.type}
    Every isolated element of $(L,\tau)$ is 
    {dually compact}. 
\end{proposition}

\begin{proof}
{Consider any $x \in L$ which is not 
{dually compact}. Since $(L,\le)$ is a dual algebraic lattice, there is a non-increasing net $\boldsymbol{x} : D \rightarrow \mathcal{K}(L,\le)$ such that $\bigwedge_d \boldsymbol{x}_d = x$. Since $\tau$ is order-compatible, we have $\boldsymbol{x}_d \rightarrow x$. Since $x$ is not dually compact, for all $d \in D$, we have $\boldsymbol{x}_d > x$. This implies that $x$ is not isolated.}
\end{proof}



\begin{proposition}\label{prop.inf.type}
    Assume that $(H,\le)$ is an upper semilattice. For all $x \in L$, if $\mathcal{M}_H(x)$ is infinite, $x$ is not isolated.  
\end{proposition}

\begin{proof}
   {Let us assume that $\mathcal{M}_H(x)$ is infinite. By Lemma \ref{lemma.cover.complements}, we have 
   \[x = \mu_H(x) \vee \left(\bigvee_{m \in \mathcal{M}_H(x)} (x-m)\right).\]
   The elements $x_P := \mu_H(x) \vee \left(\bigvee_{m \in P} (x-m)\right)$,
    where $P$ is a finite subset of $\mathcal{M}_H(x)$, form a directed set 
    which converges to $x$. Moreover for all $P$, $x_P < x$. Indeed, assume ad absurdum that there exists $P$ such that $x_P = x$. Fix $m \in \mathcal{M}_H(x) \setminus P$. Since $x_P = x$, we must have $(x-m) \wedge \mu_H(x) = (x-m)$ or $(x-m) \wedge (x-n) = (x-m)$ for some $n \in P$, meaning $(x-m) \le \mu_H(x) \le m$ or $x-m \le x-n$, both impossible.
    }
\end{proof}


{We now prove the third condition: an isolated point has no outcast.}

\begin{proposition}\label{prop.no.outcast.rev}
    Assume that $\mathcal{K}(L,\le)$ has a choice function. Every $x \in L$ which has an outcast is not isolated.
\end{proposition}

\begin{proof}
     Following Hartogs lemma, there is an ordinal $\kappa$ such that 
    there is no injection from $\kappa$ to $\mathcal{K}(L,\le)\cap (\mathop{\downarrow} x)$. Assume that $x$ has an outcast, meaning that $\partial x < x$. We define a net $\boldsymbol{s} : \kappa \rightarrow \mathcal{K}(L,\le)\cap (\mathop{\downarrow} x)$ by transfinite induction. Denote by $\theta$ a choice function of $\mathcal{K}(L,\le)$. We set $\boldsymbol{s}_0 := \partial x$ and 
    for all ordinal $\alpha < \kappa$, if $D_{\alpha} := \{z \in \mathcal{K}(L,\le) \cap (\downarrow x) \setminus \{x\} : \forall \lambda < \alpha, z > \boldsymbol{s}_{\lambda}\}$ is not empty, we set 
    $\boldsymbol{s}_{\alpha} := \theta(D_{\alpha})$,
    otherwise $\boldsymbol{s}_{\alpha} := x$. There exists $\alpha < \kappa$ such that $\boldsymbol{s}_{\alpha} = x$, otherwise there would be an injection from $\kappa$ to $\mathcal{K}(L,\le)\cap (\downarrow x)$.
    We can assume it to be minimal. We claim that $\alpha$ is a limit ordinal and that $\bigvee_{\lambda < \alpha} \boldsymbol{s}_{\lambda}= x$. Since for all $\lambda < \alpha$, $\boldsymbol{s}_{\lambda} < x$, this implies that $x$ is not isolated.
    Let us prove the claim. Assume ad absurdum that $\bigvee_{\lambda < \alpha} \boldsymbol{s}_{\lambda} < x$, while $\alpha$ is a limit ordinal. In that case, 
    $\bigvee_{\lambda < \alpha} s_{\lambda}$ is a maximal element of $(\mathop{\downarrow}x) \setminus \{x\}$. 
    {Indeed, for all $y \in L$ such that 
    $\bigvee_{\lambda < \alpha} \boldsymbol{s}_{\lambda} \le y < x$,
    there exists $y \le s < x$. Since $\boldsymbol{s}_\alpha = x$, $D_{\alpha}$ is empty, which means that  
    $\bigvee_{\lambda < \alpha} \boldsymbol{s}_{\lambda} = s = y$. 
    This implies that $\partial x$ is not an outcast of $x$, which is a contradiction. 
    For similar reasonons, assuming that $\alpha$ is not a limit ordinal, $\bigvee_{\lambda < \beta} \boldsymbol{s}_{\lambda}$, where $\alpha = \beta +1$, is a maximal subelement of $x$, which is impossible. This ends the proof.} 
\end{proof}

The following is an extension of Proposition \ref{prop.no.outcast.rev} which will be useful later.

\begin{lemma}\label{lemma.f.outcast}
     Assume that $(H,\le)$ is an upper semilattice such that for all $x \in L$, $(H\cap (\mathop{\downarrow}x),\le)$ is a complete upper semilattice. Every $x \in L$ which has an $H$-outcast is not isolated.
\end{lemma}

\begin{proof}
    It is sufficient to see that if $x$ has an $H$-outcast and $\mathcal{M}(x)$ is finite, then $x$ has an outcast. We then  
    apply Proposition \ref{prop.no.outcast.rev}. 
    Let us assume that $x$ has no outcast and that $\mathcal{M}(x)$ is finite, and prove that $x$ has no $H$-outcast.
    Fix some $z \in H$ such that $z < x$. We prove that there exists some $m \in \mathcal{M}(x)$ such that $P(z,m)$: for all $z' \in H$ such that $z \le z' < x$, $z' \le m$. 
    This implies that $\bigvee_{\underset{z' \in F}{z \le z' \le m}} z'$ is a maximal $H$-subelement of $x$, and ends the proof.
    Let us denote by $m_1 , \ldots , m_l$ the elements of $\mathcal{M}(x)$. We construct a sequence 
    $(t_i)_i$ in $\llbracket 1 , l \rrbracket$ and a sequence $(z_i)_{i=0}^k$ in $H$ as follows:  $t_0 = 1$ and $z_0 = z$. For all $i$, if $P(z_i,m_{t_i})$ is not satisfied, there exists $z' \in F$ such that $z < z' < x$ with $z' \not \le m_{t_i}$. We set $z_{i+1} = z'$ and $t_{i+1}$ is the smallest $t > t_i$ such that $z' \le m_{t}$ if it exists. Otherwise we set $t_i = l$ and end the definition. If $P(z_i,m_{t_i})$ is satisfied, we simply end the definition. 
    We have $P(z_{k},m_{t_k})$, which implies $P(z,m_{t_k})$. Otherwise, $z_k < x$ while $z_k \not \le m$ for every $m \in \mathcal{M}(x)$, which is not possible since $x$ has no outcast.
\end{proof}

We can now collect the pieces into the characterization announced at the beginning of the section.

\begin{theorem}\label{thm.characterization.isolated}
Let $(L,\le)$ be a dual algebraic coframe and let $\tau$ be an order-compatible topology on $L$. An element $x\in L$ is isolated in $(L,\tau)$ if and only if {it is dually compact}, has no outcast and $\mathcal M(x)$ is finite.
\end{theorem}

{
\begin{proof}
Direction $(\Rightarrow)$ follows from Propositions \ref{isolated.finite.type}, \ref{prop.inf.type}, and \ref{prop.no.outcast.rev}. Direction $(\Leftarrow)$ is exactly Proposition \ref{prop.no.outcast}.
\end{proof}

The theorem provides the first direct connection between the Cantor-Bendixson stratification and the residual structure of the lattice. Isolation, a purely topological notion, is completely determined by the order-theoretic notions of maximal subelements and outcasts.

As an application, we recover Theorem~A of \cite{GN25} for \((\mathcal H^d,\leq)\) equipped with \(\tau_{\mathcal H}^d\). The characterization of the next Cantor-Bendixson level, \(\mathcal S_1(L,\tau)\setminus\mathcal S_2(L,\tau)\), will be carried out in the following section.}

\section{\label{section.lattice.proofs.2}Characterization of $\mathcal{S}_1(L,\tau) \setminus \mathcal{S}_2(L,\tau)$}

In the previous section we obtained a purely order-theoretic characterization of isolated points. We now turn to the second level of the Cantor-Bendixson stratification. The analysis reveals additional structural constraints, expressed in terms of 
local {``residual patterns''}.
This clarifies how higher-order isolation phenomena are encoded in the {order-theoretic} 
structure. {We begin this section with additional context definitions (Section \ref{section.add.defs}) necessary to formulate the main result. This result is described informally in Section \ref{section.outline}, where we outline the remainder of this section, devoted to its proof.}

{
\subsection{Additional context definitions\label{section.add.defs}}
}

\subsubsection{Locally convex topologies}

Recall that a subset $S$ of a poset $(L,\le)$ is said to be \textit{convex} when for all $x,z \in S$ and $y \in L$ such that $x \le y \le z$, we have $y \in S$. 

\begin{definition}
    A topology $\tau$ on $L$ is said to be \textbf{locally convex} when it has a base of convex sets.
\end{definition}

\begin{lemma}\label{lemma.lawson.loc.convex}
    The dual Lawson topology is locally convex. 
\end{lemma}

\begin{proof}
    It is sufficient to see that for all $x$, the sets $\mathop{\downarrow}x$ and $(\mathop{\downarrow}x)^c$ are convex, which is a consequence of transitivity of the relation $\le$.
\end{proof}

The topology $\tau_{\mathcal{H}}^d$ is also clearly locally convex, as the open balls form a base of convex sets. 

\subsubsection{Relative sections}

Consider $(L,\le)$ a dual algebraic coframe whose set of dually compact elements admits a choice function.

\begin{definition}
    For all $x,z \in L$ such that $x \le z$, and all $\alpha < \texttt{r}(z)$, we set $\texttt{s}_{\alpha}(x,z) = \{s \in \texttt{s}_{\alpha}(z) : s \not \le x\}$ and $\texttt{r}(x,z)$ the smallest ordinal $\alpha < \texttt{r}(z)$ such that $\texttt{s}_{\alpha}(x,z) = \emptyset$. We call \textbf{relative section} of index $x$ in $\delta(z)$ the set $\delta(x,z) := \bigcup_{\alpha} \texttt{s}_{\alpha}(x,z)$. 
\end{definition}

{The relative section essentially consist in the part of the boundary poset of \(z\) which lies outside of $x$.}


{
\subsection{Informal description of the main theorem\label{section.outline}}
}

Throughout the remainder of this section, \((L,\leq)\) is a dual algebraic coframe whose set of compact elements admits a choice function, and \(\tau\) is an order-compatible locally convex topology on \(L\). Our goal is to characterize the elements of \(\mathcal S_1(L,\tau)\setminus\mathcal S_2(L,\tau),\) that is, the points $x \in \mathcal{S}_1(L,\tau)$ which become isolated after removing the isolated points of this space. Unlike the characterization of isolated points, this description involves both the local residual structure above and below elements. 
{It can be written informally as follows. First, an element \(x\) of $\mathcal S_1(L,\tau)\setminus\mathcal S_2(L,\tau)$ has a neighborhood whose elements are comparable to \(x\) (Section \ref{section.isolated.from.above}). From there, we have two necessary conditions for an element $x \in L$ to be in $\mathcal S_1(L,\tau)\setminus\mathcal S_2(L,\tau)$: isolation from above and isolation from below---in $\mathcal S_1(L,\tau)$. Here is an informal snapshot of these two conditions: 
\begin{itemize}
\item\textbf{Isolation from above in $\mathcal S_1(L,\tau)$} (Section \ref{section.isolated.from.above}): { The point $x$ must be isolated from above in $\mathcal{S}_1(L,\tau)$, which means that there is a neighborhood of $x$ such that every $y \ge x$ in this neighborhood must be in $\mathcal{S}_0(L,\tau)$. This is translated into the following conditions.}
If it is not dually compact (which implies that it is not necessarily isolated from above in $\mathcal S_0(L,\tau)$), there is some $z > x$ isolated in $\mathcal{S}_0(L,\le)$
such that $\texttt{c}(x) = \texttt{c}(z)$, that the relative section 
consists of index $x$ in $z$ in $\omega$ finite strata and there is a subsequence of these strata such that every element of each strata is below every element of the previous ones.

\item\textbf{Isolation from below in \(\mathcal{S}_1(L,\tau)\)} (Section \ref{section.isolated.from.below}):  {The point $x$ must also be isolated from below in $\mathcal{S}_1(L,\tau)$, which means that there is a neighborhood of $x$ such that every $y \le x$ in this neighborhood must be in $\mathcal{S}_0(L,\tau)$. This is translated into the following conditions.}
If $x$ has no maximal subelement, it must have a maximal $\mathcal{T}_0(L,\le)$-subelement $t$, and every other such subelement must be below it. Furthermore, the net of finite joins of completely co-irreducible elements outside of $t$ converge to $x$ and ultimately cover every subelement of $x$. Moreover, they all differ by finitely many completely co-irreducible elements from a fixed element of the net which is dually compact. When $x$ has at least one maximal subelement and has an outcast, the same properties apply on \(x-\partial x\), the net being shifted by \(\vee\, \partial x\) to recover \(x\).
\end{itemize}





Isolation from above relies on a local rigidity property of the core operator, that locally, every element above $x$ has the same residual core as $x$. With technical additional refinements, these necessary conditions are also sufficient.}

\subsection{Isolation from above in $\mathcal{S}_1(L,\tau)$\label{section.isolated.from.above}}


{
Any element sufficiently close to some \(x\in\mathcal S_1(L,\tau)\setminus\mathcal S_2(L,\tau)\) and lying strictly above it must be isolated in $\mathcal S_0(L,\tau)$. The purpose of this subsection is to translate this topological requirement into a collection of order-theoretic conditions involving the relative boundary structure above \(x\) as described informally in Section \ref{section.outline}. 

}

We first establish that the core operator \(\texttt{c}\) is locally constant around the elements of \(\mathcal{S}_1(L,\tau)\setminus\mathcal{S}_2(L,\tau)\). {This relies on the properties of the core operator established in Section \ref{section.core.operator.properties}.}


\begin{proposition}\label{proposition.locally.constant.points}
    For all $x \in \mathcal{S}_{1}(L,\tau) \setminus \mathcal{S}_{2}(L,\tau)$, there exists $o_x \in \tau$ such that for all $z \in o_x \setminus (\mathop{\downarrow}x)$, $c(z) = c(x)$.
\end{proposition}

\begin{remark}
    {This implies that} the operator $\texttt{c}$ {is locally maximal around on the elements}
    of $\mathcal{S}_1(L,\tau) \setminus \mathcal{S}_2(L,\tau)$. Note that is is straightforwardly the case as well for elements of $\mathcal{S}_0(L,\tau) \setminus \mathcal{S}_1(L,\tau)$. Natural questions follow, that we leave aside for further exploration. What are exactly the local maxima of $\texttt{c}$? What are the points where $\texttt{c}$ is locally constant (note that this includes isolated points)? Is this set of points included in the complement of the perfect kernel, or does it contain it? By extension, notice that the operator $\texttt{c}$ is a priori not continuous. What are the points where it is continuous?
\end{remark}

\begin{proof}
    When $x$ is dually compact, this is trivial. We thus assume that $x$ is not dually compact. There exists a non-increasing net $\boldsymbol{d} : D \rightarrow \mathcal{K}(L,\le)$ such that $\boldsymbol{x}_d \rightarrow x$. Recall that $\mathcal{O}$ is a base of convex sets of $\tau$. Assume ad absurdum that $\texttt{c}(\boldsymbol{x}_d) \not \le x$ for all $d \in D$. Then 
    the net $d \mapsto \boldsymbol{z}_d := \texttt{c}(\boldsymbol{x}_d) \vee x$ converges to $x$. Indeed, for all $o \in \mathcal{O}$ such that $x \in o$, there is some $d \in D$ such that $\boldsymbol{x}_d \in o$. 
    Since $x \le \boldsymbol{z}_d \le \boldsymbol{x}_d$ and $o$ is convex, 
    $\boldsymbol{z}_d \in o$. Moreover, for all $d \in D$, $\boldsymbol{z}_d$ is not isolated, since $x$ is an outcast of $\boldsymbol{z}_d$. This contradicts the hypothesis $ x \in \mathcal{S}_{1}(L,\tau) \setminus \mathcal{S}_{2}(L,\tau)$ and we just proved that there is some $d \in D$
    such that $\texttt{c}(\boldsymbol{x}_d) \le x$. 
    Since $\boldsymbol{x}_d \in \mathcal{K}(L,\le)$, $\mathop{\downarrow}\boldsymbol{x}_d$ is an open set, and for every $z$ in this set, by Lemma \ref{lemma.lattice.op}, $\texttt{c}(x) \le \texttt{c}(z) \le \texttt{c}(\boldsymbol{x}_d) \le \texttt{c}(x)$. This finishes the proof.
\end{proof}

{The following proposition expresses formally the condition that a non dually compact element of $L$ must satisfy to be in isolated from above in $\mathcal{S}_1(L,\tau)$.}

\begin{proposition}\label{proposition.isolated.from.above.direct}
    Assume that $\mathcal{K}(L,\le)$ has a choice function. {For all $x \in \mathcal{S}_1(L,\tau) \setminus \mathcal{S}_2(L,\tau)$ which is not dually compact, there exists $z \in \mathcal{K}(L,\le)$ with $x < z$ such that: (i) $\mathcal{M}(z)$ is finite and for all $\alpha < \texttt{r}(x,z)$, $\texttt{s}_{\alpha}(x,z)$ is finite; (ii) $\texttt{c}(x) = \texttt{c}(z)$; (iii) $\texttt{r}(x,z) = \omega$; (iv) for all $k$, there exists $l > k$ such that for all $s \in \texttt{s}_{k}(x,z)$ and all $t \in \texttt{s}_{l}(x,z)$, $t \le s$; 
    (v) if $x$ has an outcast, then for all $s \in \delta(x,z)$, $\texttt{c}(x) \le s {\vee \partial x}$; (vi) (for all $s \in \delta(x,z)$, $x \not \le s$) or (for all $s \in \delta(x,z)$, $x \le s$).}
\end{proposition}

\begin{proof}
As a direct consequence of the definition of $\mathcal{S}_1(L,\tau)$ and $\mathcal{S}_2(L,\tau)$, there exists $o_0 \in \mathcal{O}$ such that $x \in o_0$ and for all $z \in o_0 \setminus \{x\}$, $z$ is isolated. By Proposition \ref{proposition.locally.constant.points}, there exists another open set $o_1 \subset o_0$ in $\mathcal{O}$ such that for all $z \in o_1$ such that $x < z$, we have $\texttt{c}(x) = \texttt{c}(z)$.
\textbf{1.} \textbf{For all $z$ dually compact in $o_1$ such that $x < z$, the stratum of index $x$ in $\delta(z)$ is infinite}.
Otherwise, $x$ would be dually compact, contradicting the hypotheses. 
Indeed, fix such $z$ and assume ad absurdum that the stratum of index $x$ in $\delta(z)$ is finite. Since $z$ has no outcast (as $z \in o_0$), there exists $z_1 \in \mathcal{M}(z)$ such that $x \le z_1$. 
If $x = z_1$, we have that $x$ is dually compact. Otherwise, since $o_0$ is convex, $z_1$ has no outcast, hence there exists $z_2 \in \mathcal{M}(z_1)$ such that $x \le z_2$. By repeating this, since the stratum of index $x$ in $\delta(z)$ is finite, we obtain a finite sequence $z = z_0, z_1, \ldots , z_n = x$ such that for all $i$, $z_{i+1} \in \mathcal{M}(z_i)$. This implies that $x$ is dually compact.
\textbf{2.} \textbf{There exists a neighborhood $o_3 \in \mathcal{O}$, $o_3 \subset o_2$ such that whenever $z \in o_3$ is dually compact with $x < z$, $\texttt{r}(x,z) = \omega$}. Assume ad absurdum that it were not the case. Then by point 1, we would have that for all $o_3$, there is $z \in o_3$ dually compact such that $x < z$, $\texttt{r}(x,z) > \omega$. Therefore $z^{(\omega)} \vee x$ is not isolated and $x < z^{(\omega)} \vee x < z$. Since $\mathcal{K}(L,\le)$ has a choice function, this would mean that there exists a net of non isolated elements converging to $x$, which is impossible.
    \textbf{3.} \textbf{For all $z \in o_3$ and all decreasing sequence $(c_k)_{k < \omega}$ such that for all $k$,  
    $c_k \in \texttt{s}_{k}(x,z)$, 
    we have $\bigwedge_{k} c_{k} \le x$}. Indeed, point 2 implies that $z^{(\omega)} \le x$ and $\bigwedge_{k} c_{k} \le z^{(l)}$ for all $l$, meaning that $\bigwedge_{n} c_{n} \le z^{(\omega)} \le x$. 
    \textbf{4.} \textbf{For all $z \in o_3$ dually compact such that $x < z$, for all $k$, $\texttt{s}_k(x,z)$ is finite.} {Fix such $z$. Since every element of $o_3$ different from $x$ is isolated and $o_3$ is convex, for all integer $l$, $z^{(l)} \vee x$ is isolated. On the other hand, 
    we have $\texttt{s}_l(x,z) \subset \texttt{s}_0(z^{(l)}\vee x)$. In order to prove this, it is sufficient to see 
    that for all $s \in \texttt{s}_l(x,z)$, 
    \[s \not \le \bigvee_{\underset{s \le z^{(l)}\vee x}{t \in \delta(z) \setminus \{s\}}}t = \partial x \vee \left(\bigvee_{\underset{}{t \in \texttt{s}_l(x,z)\setminus \{s\}}}t\right) \vee \left(\bigvee_{k > l} \bigvee_{t \in \texttt{s}_k(x,z)} t\right).\]
    If we didn't have this, using distributivity, we would get:
    \[s \le \left(\bigvee_{\underset{}{t \in \texttt{s}_l(z)\setminus \{s\}}}t\right) \vee \left(\bigvee_{k > l} \bigvee_{t \in \texttt{s}_k(x,z)} t\right).\]
    This can be rewritten as: 
    \[s \le \left(\bigvee_{\underset{}{t \in \delta(z)\setminus \{s\}}}t\right),\]
    which is not possible since $s \in \texttt{s}_l(x,z) \subset \texttt{s}_l(z)$.
    }
    This implies that for all $l$, $\texttt{s}_l(x,z)$ is finite. 
    \textbf{5.} \textbf{For all $z \in o_3$ dually compact such that $x < z$, for all $k$, 
    there exists $l > k$ such that for all $s \in \texttt{s}_{k}(x,z)$ and $t \in \texttt{s}_{l}(x,z)$, $t \le s$.} Indeed, assume ad absurdum that there exists 
    $z \in o_3$ which contradicts this condition, meaning that there exists $k$ such that for all $l > k$, there exists $s \in \texttt{s}_k(x,z)$ and $t \in \texttt{s}_l(x,z)$ such that $t \not \le s$. By point 4 and the pigeon-hole principle,
    there exists some $s_k \in \texttt{s}_k(x,z)$ and an increasing sequence $(l^k_i)_i$ of integers such that for all $i$, there is some $t_i \in \texttt{s}_{l^k_i}(z) \cap \delta(x,z)$ 
    such that $t^k_i \not \le s$. As a consequence of point 2, 
    $t^k_i \vee s_k \vee x \rightarrow s_k \vee x$. On the other hand, by distributivity, and since $t^k_i \not \le s_k$ and $t^k_i \not \le x$, we have $t^k_i \not \le s_k \vee x$. This 
    means that $t^k_i \vee s_k \vee x \neq s_k \vee x$ for all $k$ and $i$ and implies that for all $k$, $s_k \vee x$ is not isolated. For similar reasons, $s_k \vee x \rightarrow x$ while for all $k$, $s_k \vee x \neq x$. This is impossible since 
    $x \in \mathcal{S}_1(L,\tau) \setminus \mathcal{S}_2(L,\tau)$.
    \textbf{6.} \textbf{If $x$ has an outcast, there is a neighborhood $o_4 \subset o_3$ of $x$ such that for all $z \in o_4$ dually compact such that $x < z$, for all $s \in \delta(x,z)$, $\texttt{c}(x) \le s { \vee \partial x}$.} Assume ad absurdum that it is not the case, meaning that for all $o \subset o_3$ such that $x \in o$, there exists $z_o \in o$ dually compact such that $x < z$ and $s_o \in \delta(x,z_o)$ such that $\texttt{c}(x) \not \le s_o { \vee \partial x}$. In particular, $x \vee s_o > x$. By point 3,
    $\bigwedge_o s_o \le \texttt{c}(x) \le x$, which means that 
    $x \vee s_o$ converges to $x$. It is then sufficient to see that $\texttt{c}(x\vee s_o) \ge \texttt{c}(x)$ and $\partial (x \vee s_o) \le s_o \vee \partial x$, meaning that 
    $\partial (x \vee s_o) < x \vee s_o$, and that $x \vee s_o$ has an outcast. This contradicts again the hypothesis that 
    $x \in \mathcal{S}_1(L,\tau) \setminus \mathcal{S}_2(L,\tau)$.
    If $x$ has no outcast, we set $o_4 := o_3$.
    \textbf{7.} \textbf{There exists $z \in o_4$ dually compact such that $x < z$ and (for all $s \in \delta(x,z)$, $x \not \le s$) or (for all $s \in \delta(x,z)$, $x \le s$).} 
    Indeed, fix $y \in o_4$. If there is some $n$ such that 
    for all $s \in \texttt{s}_n(x,y)$, $x \not \le s$, then $z:= y^{(n)} \vee x \in o_4$ satisfies the condition. Otherwise for all $n$ there is some $s_n \in \texttt{s}_n(x,y)$ such that $x \le s$. As a consequence of point 5, $z:=y$ itself satisfies the condition. 
\end{proof}

\begin{remark}
    Assume that there exists a dually compact element $z$ such that for all $n$, $\mathcal{M}(z^{(n)})$ has a unique element. Then $z_0$ satisfies the conditions of Proposition \ref{proposition.isolated.from.above.direct} for $x = z^{(\omega)}$. Note that when $x \in (\mathcal{S}_1(L,\tau) \setminus \mathcal{S}_2(L,\tau)) \cap \mathcal{T}_0(L,\le)$,
    $x$ is dually compact, or $x = z^{(\omega)}$ for $z$ such that for all $n$, $\mathcal{M}(z^{(n)})$ has a unique element.
\end{remark}

{
We now show that the previous condition is also sufficient for isolation from above in $\mathcal{S}_1(L,\tau)$.}


\begin{proposition}\label{proposition.comparable.neighborhood}
    {If $x \in L$ satisfies the conditions of Proposition \ref{proposition.isolated.from.above.direct}, there is a neighborhood $o\in \tau$ of $x$ such that 
     for all $y \in o$, $x \le y$ or $y \le x$.}
\end{proposition}

\begin{proof}
        Fix $z$ provided by Proposition \ref{proposition.isolated.from.above.direct}. Let us first prove that there is a neighborhood $o \subset \mathop{\downarrow}z$ of $x$ such that for all $y \in o$, $y \le x$ or $x \le y$. 
    By Lemma \ref{lemma.subelement.decomposition}, for every $y \in \mathop{\downarrow}z$ such that $\texttt{c}(y) = \texttt{c}(x)$, 
    \[y = \texttt{c}(x) \vee \left(\bigvee_{\underset{s \le y}{s \in \delta(z)}} s \right).\]
    If for all $s \in \delta(x,z)$, $x \le s$, then 
    we get from this formula that else $y \le x$ or $x \le y$. Otherwise, by the property (vi), for all $s \in \delta(x,z)$, $x \not \le s$. We thus have $\texttt{c}(x) < x$, as otherwise we would have $\partial x = \varepsilon$, and, by property (v), $x \le s$ for all $s \in \delta(x,z)$.
    Since $\texttt{c}(z) = \texttt{c}(x)$ (property (ii)) there exists $k$ such that $z^{(k)} \wedge x < x$. For all $l$, 
    we set $s_l := \bigvee_{s \in \texttt{s}_l(x,z)} s$. Since $\bigwedge_l s_l \le \texttt{c}(x)$, there exists $l \ge k$ such that for all $m \in \mathcal{M}(x)$, $x \not \le m \vee s_l$. Assume ad absurdum that for all $o$ such that $x \in o$, we have some $y_o \in o$ such that $y_o \not \le x$ and $x \not \le y_o$. The element $y_o$ can be chosen as $\theta(\{y : y \in o \cap \mathcal{K}(L,\le) , x \not \le y , y \not \le x\})$. For all $o \subset \mathop{\downarrow}(z^{(l)}\vee x)$, $y_o \le m \vee s_l$ for some $l$. 
    Since the net $(y_o)_o$ converges to $x$ and $\le$ is closed, this is impossible. 
\end{proof}
{
\begin{remark}
Observe that as a consequence of Proposition~\ref{proposition.comparable.neighborhood}, whether a point \(x\) is isolated in \(\mathcal{S}_1(L,\leq)\) depends only on the elements immediately above and below \(x\) and not on elements incomparable to \(x\). This allows us to fully characterize isolated points in \(\mathcal{S}_1\) in terms of their isolation from above and below. 
\end{remark}
}
\begin{proposition}\label{proposition.reverse.from.above}
    If $x \in L$ satisfies the conditions of Proposition \ref{proposition.isolated.from.above.direct}, there is a neighborhood $o\in \tau$ of $x$ such that 
    for all $y \in o$, if $x < y$, then $y$ is isolated.
\end{proposition}

\begin{proof}
    Fix $z$ provided by Proposition \ref{proposition.isolated.from.above.direct}. We prove that for all $y \in o$ such that $x < y$, $y$ is isolated. By Lemma \ref{lemma.subelement.decomposition} we have: 
\[y = \texttt{c}(z) \vee \bigvee_{\underset{t \le y}{t \in \delta(z)}} t.\]
    Since $y > x$, there exists $s \in \delta(x,z)$ such that $s \le y$. Let us denote by $k$ the integer such that $s \in \texttt{s}_k (x,z)$. We know that there exists $l>k$ such that for all $t \in \texttt{s}_{l}(x,z)$, $t \le s$. This implies that there exists a finite set $P \subset \delta(z)$ such that 
    \[y = \texttt{c}(z) \vee \bigvee_{{t \in \delta(z) \setminus P}} t.\]
    By Corollary \ref{corollary.submax.dominant}, since $z$ is dually compact, $y$ is also dually compact.
    In the case $x \le s'$ for all $s' \in \delta(x,z)$, then $\mathcal{M}(y)$ has finitely many elements and $\texttt{c}(y) = \texttt{c}(x) \le \mu(y)$. This implies that $y$ is isolated. Assume that for all $s \in \delta(x,z)$, $x \not \le s$. This means that $\mathcal{M}(z)$ has at least two elements. Otherwise, $\texttt{s}_{0}(z)$ would have a unique element which would be also in $\delta(x,z)$. Since $z$ has no outcast, $z$ is equal to this element which thus contains $x$, contradicting the hypothesis. For similar reasons, there is an element $m \in \mathcal{M}(z)$ such that $z-m \le x$. 
    From Lemma \ref{lemma.subelement.decomposition}, we can write: 
     \[y = (s_1 \vee \ldots \vee s_k) \vee \left(\bigvee_{s' \in \texttt{s}_0(z) \cap \delta(x,z)} s' \right),\]
     for some $s_1, \ldots , s_k \in \delta(x,z)$. 
     Then every $y' < y$ is smaller than 
     \[(s_1 \vee \ldots \vee \mu(s_i) \vee \ldots \vee s_k)  \vee \left(\bigvee_{s' \in \texttt{s}_0(z) \setminus \{s\}} s' \right) \quad \text{or} \quad (s_1 \vee \ldots \vee s_k) \vee \left(\bigvee_{s' \in \texttt{s}_0(x,z) \setminus \{s''\}} s' \right),\]
     where $s'' \neq s$. This implies that $y$ has finitely many subelements and has no outcast, thus $y$ is isolated.
\end{proof}

{
Together, the previous propositions provide a complete description of the residual configurations above a point that are compatible with membership in \(\mathcal S_1(L,\tau)\setminus\mathcal S_2(L,\tau)\).
}

\subsection{Isolation from below in $\mathcal{S}_1(L,\tau)$\label{section.isolated.from.below}}


{
Any element sufficiently close to some \(x\in\mathcal S_1(L,\tau)\setminus\mathcal S_2(L,\tau)\) and lying strictly below it must be isolated in $\mathcal S_0(L,\tau)$. The purpose of this subsection is to translate this topological requirement into a collection of order-theoretic conditions that involve the relative boundary structure above \(x\) as described informally in Section \ref{section.outline}. 
The analysis splits naturally according to whether \(x\) belongs to \(\mathcal T_0(L,\leq)\) or not. 

}
\begin{lemma}
    {Let $(L,\le)$ be a dual algebraic lattice and let $x \in L$. Then $\mathop{\downarrow}x$ is compact for the topology $\tau^*(L)$.}
\end{lemma}

\begin{proof}
    Fix $x \in L$ and $(o_i)_{i \in I}$ an open cover of $\mathop{\downarrow}x$ in the topology $\tau^*(L)$ such that for all $i$, there exists a dually compact element $z_i \in L$ and $F_i$ a finite set of dually compact elements in $\mathop{\downarrow}x$ such that $o_i = \mathop{\downarrow}z_i \setminus \left(\bigcup_{z \in F_i} \mathop{\downarrow}z \right)$. Assume ad absurdum that no finite subfamily of $(o_i)_{i \in I}$ covers $\mathop{\downarrow}x$, meaning that for every finite $J \subset I$, there exists $y_J \in \mathop{\downarrow}x \cap \mathcal{K}(L,\le)$ such that $y_J \notin \bigcup_{j \in J} o_j$. Let us set $y := \bigvee_{\underset{J finite}{J \subset I}} y_J$. 
    Since $y \le x$, there exists $i$ such that $y \in o_i$. The net $(y_J)_{\underset{J finite}{J \subset I}}$ converges to $y$, so there exits $J$ such that for all $J' \supset J$, $y_{J'} \in o_i$. In particular there exists $J' \subset I$ finite such that $i \in J'$ and $y_{J'} \in o_i \subset \bigcup_{j \in J'} o_j$, which is impossible, because by definition $y_{J'} \notin \bigcup_{j \in J'} o_j$.
\end{proof}

It is also clear that for the topology $\tau_{\mathcal{H}}^d$, for all $x \in \mathcal{H}^d$, $\mathop{\downarrow}x$ is compact. \bigskip 

For all $x \in L$, if $\mathcal{M}(x)$ is finite and $x$ has no outcast, there exists some $o \in \tau$ such that 
$x \in o$ and for all $y \in o$, $y \ge x$. We consider the other cases below, meaning when $x$ has an outcast.


\begin{proposition}\label{proposition.isolated.below.direct}
    Assume that $\mathcal{K}(L,\le)$ has a choice function. Consider some $x \in \mathcal{S}_1(L,\tau) \setminus \mathcal{S}_2(L,\tau)$ such that $x \in \mathcal{T}_0(L,\le)$. The set $\mathcal{M}_{\mathcal{T}_0(L,\le)}(x)$ has a unique element, $x$ has no $\mathcal{T}_0(L,\le)$-outcast.
    Denote by $\delta x$ the set $(\mathop{\downarrow}x) \cap \mathcal{I}(L,\le) \setminus (\mathop{\downarrow} \mu_{\mathcal{T}_0(L,\le)}(x))$. The net $\boldsymbol{h} : P \subset \delta x \mapsto \bigvee_{s \in P} s$ satisfies the following conditions: 
    (i) for all $P$, $\boldsymbol{h}_P < x$ and $\bigvee_{P} \boldsymbol{h}_P = x$; 
    (ii) for all $z < x$, there exists $P$ such that $z \le \boldsymbol{h}_P$; (iii) there exists $P^*$ such that $\boldsymbol{h}_{P^*}$ is isolated and $\texttt{c}(\boldsymbol{h}_{P^*}) = \mu_{\mathcal{T}_0(L,\le)}(x)$; (iv) for all $P \supset P^*$, $\boldsymbol{h}_P$ is dually compact; 
    (v) for all $P \supset P^*$, 
    $\delta(\boldsymbol{h}_P,\boldsymbol{h}_{P^*})$ is finite; (vi) for all $o \in \tau$ such that $x \in o$, there exists $P \supset P^*$ such that for all $(\delta x \setminus P) \subset o$; (vii) for all $s \in \delta x \setminus P^*$, $\delta(s \wedge \boldsymbol{h}_{P^*} , \boldsymbol{h}_{P^*})$ is finite, $\mu_{\mathcal{T}_0(L,\le)}(x) \le s$, and $s \wedge \boldsymbol{h}_{P^*}$ has finitely many subelements.
\end{proposition}

\begin{proof}
\textbf{1.} The facts that $\mathcal{M}_{\mathcal{T}_0(L,\le)}(x)$ is finite and that $x$ has no $\mathcal{T}_0(L,\le)$-outcast come from the fact that $(\mathcal{T}_0(L,\le) \cap (\mathop{\downarrow x}),\le)$ is a complete upper semilattice (Lemma \ref{lemma.t0.uppersemilattice}) and Proposition \ref{prop.inf.type} and Lemma \ref{lemma.f.outcast}. Let us prove that $\mathcal{M}_{\mathcal{T}_0(L,\le)}(x)$ has a unique element. Assume ad absurdum that it has at least two distinct elements, that we denote by $m,m'$. The element $m \wedge m'$ is an outcast both for $m$ and $m'$, because $m,m' \in \mathcal{T}_0(L,\le)$. This implies that 
there are two nets $\boldsymbol{z}: D \rightarrow \mathop{\downarrow}m \setminus \{m\}$ and $\boldsymbol{z}' : D' \rightarrow \mathop{\downarrow}m' \setminus \{m'\}$ such that $\boldsymbol{z}_d \rightarrow m$, $\boldsymbol{z}'_d \rightarrow m'$ with $m \wedge m' < \boldsymbol{z}_d$ and $m \wedge m' < \boldsymbol{z}'_d$ for all $d$.
For all $d$, we have $\boldsymbol{z}_d \vee \boldsymbol{z}'_{d'} \rightarrow \boldsymbol{z}_d \vee m$, while for all $d,d'$, $\boldsymbol{z}_d \vee \boldsymbol{z}'_{d'} < \boldsymbol{z}_d \vee m'$. 
Indeed, if there were some $d,d'$ such that $\boldsymbol{z}_d \vee \boldsymbol{z}'_{d'} = \boldsymbol{z}_d \vee m'$, we would have $m' \wedge (\boldsymbol{z}_d \vee \boldsymbol{z}'_{d'})= m' \wedge (\boldsymbol{z}_d\vee m')$, meaning, since $\boldsymbol{z}_d \wedge m' = m \wedge m'$, that 
$(m \wedge m') \vee (m' \wedge \boldsymbol{z}'_{d'}) = (m\wedge m') \vee m' = m'$, which can be rewritten 
as $\boldsymbol{z}'_{d'} = m'$, which is impossible.
We have proved that $\boldsymbol{z}_d \vee m'$ is not isolated. Similarly, $\boldsymbol{z}_d \vee m' \rightarrow m \vee m' = x$, which contradicts the hypothesis $x \in \mathcal{S}_1(L,\tau) \setminus \mathcal{S}_2(L,\tau)$.
 \textbf{2.} There exists an element $o^*$ of a convex base of $\tau$ such that $x \in o^*$ which, except for $x$, contains only isolated elements, and in particular no element of $\mathcal{T}_{\infty}(L)$. By Proposition \ref{proposition.comparable.neighborhood}, we can assume that for all $y \in o^*$, $y \le x$ or $x \le y$. Since $x \in \mathcal{T}_0(L,\le)$, there exists $z \in o^* \cap (\mathop{\downarrow}x \setminus \{x\})$ such that $\mu_{\mathcal{T}_0(L,\le)}(x) \le z$. There exists a non-increasing net $\boldsymbol{x} : D \rightarrow \mathcal{K}(L,\le)$ such that $\boldsymbol{x}_d \rightarrow z$.
We can assume that for all $d$, $\boldsymbol{x}_d \in o^*$. 
On the other hand, it is not possible that for all $d$, $x \le \boldsymbol{x}_d$, since $z < x$. This implies that there exists $d_0 \in D$ such that $\boldsymbol{x}_{d_0} < x$. 
For all $f \in \mathcal{K}(L,\le)$ such that $f < x$, $\texttt{c}(f) \le \mu_{\mathcal{T}_0(L,\le)}(x)$.
If we also have $f \ge \mu_{\mathcal{T}_0(L,\le)}(x)$, 
then $\texttt{c}(f) = \mu_{\mathcal{T}_0(L,\le)}(x)$.
Since $x$ has no outcast, there exists 
a non-decreasing net $\boldsymbol{z} : D \rightarrow \mathop{\downarrow}x\setminus \{x\}$, such that for all $d$, $\boldsymbol{x}_{d_0} \le \boldsymbol{z}_{d}$ and $\boldsymbol{z}_{d} \rightarrow x$. For all $d \in D$, we have $\texttt{c}(\boldsymbol{z}_d) = \mu_{\mathcal{T}_0(L,\le)}(x)$. 
Since $o^*$ is convex, for all $d$, $\boldsymbol{z}_d$ is isolated. In particular, for all $d,d' \in D$ such that $d < d'$, we have that $\delta(\boldsymbol{z}_d,\boldsymbol{z}_{d'})$ is finite. 
In particular, for all $d \ge d_0$, there exists $P \in \delta x$ finite such that $\boldsymbol{z}_d = \boldsymbol{h}_P$. We thus have $\boldsymbol{h}_P \rightarrow x = \bigvee_{P} \boldsymbol{h}_P$. Since $x \in \mathcal{T}_0(L,\le)$, for all $P$, $\boldsymbol{h}_P < x$. We just proved condition (i).
\textbf{3.} Denote by $P^*$ finite subset of $\delta x$
such that $\boldsymbol{x}_{d_0} = \boldsymbol{h}_{P^*}$.
Since $o^*$ is convex, for all $P \supset P^*$, 
we have $\boldsymbol{h}_P \in o^*$, which means 
that $\boldsymbol{h}_P$ is isolated. This implies both condition (iii) and (iv). Since every element $z$
such that $\boldsymbol{h}_{P^*} \le z \le \boldsymbol{h}_P$ is isolated, we have $\delta(\boldsymbol{h}_P,\boldsymbol{h}_{P^*})$ is finite (condition (v)). 

\textbf{3.} {If condition (vi) were not true, there would exist some neighborhood $o \in \tau$ of $x$, 
such that for all $P$ finite such that $P \supset P^*$, 
there exists $s \in \delta x \setminus P$ such that 
$s \notin o$. We set $s_P := \theta((\mathcal{D}(L) \cap (\delta x \setminus P)) \setminus o)$.
By compactness of $\mathop{\downarrow}x$, 
there exists a directed subset $E$ of the set of finite subsets of $\delta x$ such that 
$P \in E \mapsto s_P$ converges to some $s \le x$.
Since $o$ is open, we have $s < x$. Moreover, there exists $P$ such that 
$\boldsymbol{h}_P \vee s = x$. Indeed, for all $P \subset \delta x$ finite, since $\boldsymbol{h}_P$ is dually compact, and for all $Q \supsetneq P$, $s_{Q} \not \le \boldsymbol{h}_P$, 
we must have $s \not \le \boldsymbol{h}_{P}$. If for all $P$ we had $\boldsymbol{h}_P \vee s < x$, we would have a net of non-isolated elements in $\mathop{\downarrow x} \setminus\{x\}$ converging to $x$. It is also not possible that $s \in \mathcal{T}_0(L,\le)$, otherwise we would have $s \le \mu_{\mathcal{T}_0(L,\le)}(x)$ since $x$ has not $\mathcal{T}_0(L,\le)$-outcast. Thus $s\le \boldsymbol{h}_P$ and consequently $\boldsymbol{h}_P = x$, which is false. 
We just proved that $\mathcal{M}(s)$ has at least one element. Since $\texttt{c}(s) \le \mu_{\mathcal{T}_0(L,\le)}(x)$ and $\mu_{\mathcal{T}_0(L,\le)}(x) \le \boldsymbol{h}_P$, which is isolated, it not possible that for all $m \in \mathcal{M}(s)$, $s-m \le \boldsymbol{h}_P$. Otherwise we would have $s \le \boldsymbol{h}_P$. Set $m \in \mathcal{M}(s)$ such that $s-m \not \le \boldsymbol{h}_P$. Then $\boldsymbol{h}_P \vee m$ is a maximal subelement of $x$, which is not possible.
}
\textbf{4.} In order to prove condition (ii), 
assume ad absurdum that there exists $z < x$ such that 
for all $P \supset P^*$, $z \not \le \boldsymbol{h}_{P}$.
This implies that $z \vee \boldsymbol{h}_{P} = x$ or 
$\delta(z,z \vee \boldsymbol{h}_{P}$ is infinite.
It is not possible that for all $P\supset P^*$ there exists $Q \supset P$ such that $z \vee \boldsymbol{h}_{Q} < x$, since $x \in \mathcal{S}_1(L,\tau) \setminus \mathcal{S}_2(L,\tau)$. Therefore, there exists 
$P \supset P^*$ such that $z \vee \boldsymbol{h}_{P} = x$. Since for all $P$, $\boldsymbol{h}_P < x$, there is some $s \in \delta x \setminus P$, $s \le z$. As a consequence of condition (vi) and by continuity, $z \ge x$, which contradicts hypothesis $z < x$. 
\textbf{5.} We prove (vii) by reasoning ad absurdum. 
Assuming that for all $P \supset P^*$ there exists
$s_P \in \delta x \setminus P$ such that  $\mu_{\mathcal{T}_0(L,\le)}(x) \not \le s$, we would have 
a net of non-isolated elements different from $x$ converging to it. There is thus some $P^{0}$ 
such that for all $P \supset P^{0}$, $\mu_{\mathcal{T}_0(L,\le)}(x) \le s$. Assume that 
for all $P \supset P^{0}$, there exists $s_P \in \delta x \setminus P$ such that $\delta(s_P \wedge \boldsymbol{h}_{P^*},\boldsymbol{h}_{P^*})$ is infinite. This implies again the existence of a net of non-isolated elements different from $x$ converging to it.
We proved that there exists $P^{1} \supset P^0$ such that for all $s \in \delta x \setminus P^1$, $\delta(s \wedge \boldsymbol{h}_{P^*},\boldsymbol{h}_{P^*})$ is finite. For similar reasons there exists $P^2 \supset P^1$ such that for all $s \in \delta x \setminus P^2$, 
$s \wedge \boldsymbol{h}_{P^*}$ has finitely many maximal subelements. As a consequence of condition (v), we also have that for all such $s$, 
$\delta(s \wedge \boldsymbol{h}_{P^2},\boldsymbol{h}_{P^2})$ is finite. 
As a consequence of condition (iii) and (v), the condition (iii) is true also when replacing $P^*$ with $P^1$. The other conditions are trivially satisfied replacing $P^*$ with $P^2$. This ends the proof.
\end{proof}

\begin{proposition}\label{proposition.reverse.isolated.below.1}
    Consider $x \in \mathcal{T}_0(L,\le)$ 
    which satisfies the conditions of Proposition \ref{proposition.isolated.below.direct}. There exists a neighborhood $o \in \tau$ of $x$ such that 
    $o \cap (\mathop{\downarrow}x) \setminus \{x\}$ contains only isolated elements of $L$. 
\end{proposition}

\begin{proof}
    Since for all $z < x$, there exists some $P$ such that $z \le \boldsymbol{h}_P$ (condition (ii)),
    every $z < x$ such that $z \not \le \boldsymbol{h}_{P^*}$ satisfies $z \le \boldsymbol{h}_{P}$ for some $P \supsetneq P^*$. 
    In this case, there is some $s \in \delta x \setminus P^*$ such that $s \le z$. As a consequence 
    of condition (vii), $\delta(z \wedge \boldsymbol{h}_{P^*},\boldsymbol{h}_{P^*}) \le \delta(s \wedge \boldsymbol{h}_{P^*},\boldsymbol{h}_{P^*})$ is finite. This implies that $z$ is dually compact by conditions (iv) and (v): $\delta(z)$ is finite if and only if . 
    Since $\mu_{\mathcal{T}_0(L,\le)}(x) \le s$, it has no outcast. Furthermore, for all $s \in \delta x \setminus P^*$, $s \wedge \boldsymbol{h}_{P^*}$ has finitely many maximal subelements, thus $z\wedge \boldsymbol{h}_{P^*}$ and $z$ also satisfy this property, by Lemma \ref{lemma.subadditivity.type}. We have thus proved that for $z < x$ close enough to $x$, $z$ is isolated.
\end{proof}

The following two propositions are a technical adaptation of the two last propositions. We omit the proof in order to not overcharge the article with technical details. 

\begin{lemma}\label{coheyting.in.T0}
    For all $x \in L$, we have $x - \partial x \in \mathcal{T}_0(L,\le)$. In particular, $x - \partial x \le \texttt{c}(x)$. 
\end{lemma}

\begin{proof}
    If $x \in \mathcal{T}_0(L,\le)$, this is straightforward.
    Assume that $x \notin \mathcal{T}_0(L,\le)$.
    If we had $\mathcal{M}(x-\partial x) \neq \emptyset$, 
    we would have $((x-\partial x) - m) \le \partial x$ (otherwise $m \notin \mathcal{M}(x)$, which is obviously not true), 
    and thus $m \vee \partial x = x$, which is not possible, by definition of $x - \partial x$. The second part of the statement comes from the following: if we had two distinct elements $m,n \in \mathcal{M}_{\mathcal{T}_0(L,\le)}(x-\partial x)$, 
    we would have $m \vee n = x - \partial x$. 
\end{proof}

\begin{proposition}\label{proposition.isolated.below.nont0.direct}
    Assume that $\mathcal{K}(L,\le)$ has a choice function. Consider some $x \in \mathcal{S}_1(L,\tau) \setminus \mathcal{S}_2(L,\tau)$ such that $x \notin \mathcal{T}_0(L,\le)$ and $x$ has an outcast. The set $\mathcal{M}_{\mathcal{T}_0(L,\le)}(x-\partial x)$ has a unique element, $x$ has no $\mathcal{T}_0(L,\le)$-outcast, 
    and $\mu_{\mathcal{T}_0(L,\le)}(x-\partial x) \vee \partial x < x$ while $\mu_{\mathcal{T}_0(L,\le)}(x-\partial x) \not \le \partial x$.
    Denote by $\delta x$ the set $(\mathop{\downarrow}x) \cap \mathcal{I}(L,\le) \setminus (\mathop{\downarrow} (\mu_{\mathcal{T}_0(L,\le)}(x-\partial x))\vee \partial x)$. The net $\boldsymbol{h} : P \subset \delta x \mapsto \left(\bigvee_{s \in P} s\right) \vee \partial x$ satisfies the following conditions: 
    (i) for all $P$, $\boldsymbol{h}_P < x$ and $\bigvee_{P} \boldsymbol{h}_P = x$; 
    (ii) for all $z < x$, there exists $P$ such that $z \le \boldsymbol{h}_P$; (iii) there exists $P^*$ such that $\boldsymbol{h}_{P^*}$ is isolated and $\texttt{c}(\boldsymbol{h}_{P^*}) = \mu_{\mathcal{T}_0(L,\le)}(x-\partial x) \vee \texttt{c}(\partial x)$; (iv) for all $P \supset P^*$, $\boldsymbol{h}_P$ is dually compact; 
    (v) for all $P \supset P^*$, 
    $\delta(\boldsymbol{h}_P,\boldsymbol{h}_{P^*})$ is finite; (vi) for all $o \in \tau$ such that $(x-\partial x) \in o$, there exists $P \supset P^*$ such that for all $(\delta x \setminus P) \subset o$; (vii) for all $s \in \delta x \setminus P^*$, 
    the set $\{t \in \delta x : t \le \boldsymbol{h}_{P^*}, t \not \le s\}$ is finite, $\mu_{\mathcal{T}_0(L,\le)}(x-\partial x) \le s \vee \partial x$, and $s \wedge \boldsymbol{h}_{P^*}$ has finitely many subelements.
\end{proposition}

\begin{proposition}\label{proposition.reverse.last}
    Consider $x \in \mathcal{T}_0(L,\le)$ 
    which satisfies the conditions of Proposition \ref{proposition.isolated.below.nont0.direct}. There exists a neighborhood $o \in \tau$ of $x$ such that 
    $o \cap (\mathop{\downarrow}x) \setminus \{x\}$ contains only isolated elements of $L$. 
\end{proposition}


\subsection{{Synthesis}}

{
We now have all the ingredients necessary to 
formulate our main result, which is essentially a combination of the results in the previous two subsections.} As the unpacked formulation is quite extensive, instead of rewriting them, we refer directly to the respective propositions. 


\begin{theorem}
Consider a dual algebraic coframe $(L,\le)$ such that $\mathcal{K}(L,\le)$ has a choice function, and $\tau$ an order-compatible topology on $L$ which is locally convex. 
An element \(x\in L\) is in \(\mathcal{S}_1(L,\tau)\setminus \mathcal{S}_2(L,\tau)\) if and only if: 
\begin{enumerate}
\item[(i)] $x$ is dually compact or satisfies the conditions of Proposition~\ref{proposition.isolated.from.above.direct}; 
\item[(ii)] $x$ satisfies the conditions of Proposition \ref{proposition.isolated.below.direct}
when $x \in \mathcal{T}_0(L,\le)$ and it satisfies the conditions of Proposition \ref{proposition.isolated.below.nont0.direct} when $x \notin \mathcal{T}_0(L,\le)$.
\end{enumerate}
\end{theorem}

\begin{proof}
    Direction $(\Rightarrow)$ is straightforward from the cited propositions. We detail direction $(\Leftarrow)$. Consider $x \in L$ that satisfies (i) and (ii). If $x$ is dually compact, by definition it has a neighborhood contained in $\mathop{\downarrow}x$. Otherwise, by condition (i), $x$ satisfies the conditions of Proposition~\ref{proposition.isolated.from.above.direct}. In this case, $x$ has a neighborhood of points that are all below or above $x$, and $x$ is isolated from above in $\mathcal{S}_1(L,\tau)$, by Proposition \ref{proposition.comparable.neighborhood} and Proposition \ref{proposition.reverse.from.above}.In both cases, it suffices to see that it is isolated from below in $\mathcal{S}_1(L,\tau)$, which is ensured by conditions (ii) and Propositions \ref{proposition.reverse.isolated.below.1} and \ref{proposition.reverse.last}. 
\end{proof}

{Schematically, condition (i) ensures that the element $x$ is isolated from above in $\mathcal{S}_1(L,\tau)$, and condition (ii) ensures that it is isolated from below in $\mathcal{S}_1(L,\tau)$.}


The characterization just established is considerably more involved than its counterpart for isolated points. Whereas membership in \(\mathcal{S}_0(L,\tau)\setminus\mathcal{S}_1(L,\tau)\) depends only on the local residual structure below \(x\)---captured entirely by the finiteness of \(\mathcal{M}(x)\) and the absence of any outcast---membership in \(\mathcal{S}_1(L,\tau)\setminus\mathcal{S}_2(L,\tau)\) requires coordinating two independent families of conditions: one on the behaviour of the core operator and the strata above \(x\), and one on the directed systems of isolated elements accumulating below it. The asymmetry between these two directions reflects the fact that \(\tau\) is order-compatible but not symmetric.


The fact that both directions can nonetheless be expressed purely in order-theoretic terms---without reference to specific topological data beyond the coframe structure---confirms that the Cantor-Bendixson stratification, at least at this level, is an intrinsic feature of \((L,\leq)\) rather than of the choice of order-compatible topology. Whether this remains true at higher levels of the stratification is a natural question, which the methods developed here do not directly resolve.

\section{\label{section.comments}Comments}

The results obtained in this work show that the Cantor–Bendixson hierarchy, although defined topologically, can be reconstructed from purely order-theoretic structure in a broad class of structures. This suggests that similar correspondences may hold in other ordered settings. In particular, it raises the question of identifying minimal conditions under which a topological derivative can be represented by an intrinsic order-theoretic operator.

We formulate below some questions that have arisen during the course of this work, providing natural directions for further research.

\paragraph{Realization}
The characterization of $\mathcal{S}_1(L,\tau)\setminus \mathcal{S}_2(L,\tau)$ reveals that this set can be 
divided into elements are isolated from above, ones that are isolated from below and others. A natural question is whether each of them is non-trivial, in particular for our motivating example: 
\begin{question}
We know that there exist elements of $\mathcal{S}_1(\mathcal H^d,\tau_{\mathcal{H}}^d)\setminus \mathcal{S}_2(\mathcal H,\tau_{\mathcal{H}}^d)$ which are isolated from above, as they are dually compact (ie shifts of finite type). On the other hand, does there exist a non-finite-type shift (non isolated from above) in $\mathcal{S}_1(\mathcal H^d,\tau_{\mathcal{H}}^d)\setminus \mathcal{S}_2(\mathcal H,\tau_{\mathcal{H}}^d)$?
\end{question}
More generally, one may ask whether elements in this family can exhibit isolation from below, and how such properties can be detected in terms of residual structure.
\paragraph{Special subsets}
The objects introduced in this work - such as the core $\texttt{c}(x)$ and the boundary $\partial x$ - define natural subfamilies of elements. For instance, one may consider elements with trivial core, or those for which the decomposition $x = \texttt{c}(x) \vee \partial x$ is degenerate in a prescribed way.
A systematic study of these subfamilies could refine the understanding of the topological structure of the lattice, in a similar way as characterizing lattices in which certain of these subsets are empty.

\paragraph{Collapse of Cantor–Bendixson hierarchy}

The interaction between residual and Cantor–Bendixson derivatives raises the question of how Cantor–Bendixson levels can degenerate. Having characterized the first levels, it is natural to ask when some of them collapse entirely. For instance:
\begin{question}
Which dual algebraic coframes satisfy $\mathcal{T}_0(L,\le)={\varepsilon}$?
\end{question}
In such cases, non-isolated elements coincide with $\mathcal{T}_{\infty}(L)$, and under additional assumptions (see below), the Cantor–Bendixson rank is 0 or 1. These extremal cases provide simple models in which the relationship between residual structure and Cantor–Bendixson decomposition becomes particularly transparent.
\paragraph{Perfect kernel}
When $\mathcal{K}(L,\le)$ is countable, the topology becomes first countable, and one obtains partial information over the perfect kernel. In particular, $\mathcal{T}_{\infty}(L,\le) \subset \mathcal{S}_{\texttt{r}(L,\tau)}(L,\tau)$. It would be of interest to determine to what extent this relation depends on the countability assumption.
\medskip

Another natural research direction consists in interpreting existing results computing the Cantor-Bendixson rank of algebraic structures, such as the Grigorchuk group \cite{skipper2020cantorbendixsonrankgrigorchukgroup} in terms of our framework. Furthermore, it would be interesting to see how the Cantor-Bendixson structure varies depending on the order on the same set. For instance, in the hyperspace of shifts, Nathalie Aubrun and Mathieu Sablik \cite{AubrunSablik2009} defined an order different from the inclusion, based on the algorithmic enumerability of the language.

\bibliographystyle{plain}   
\bibliography{biblio.bib}

\end{document}